\documentclass[11pt]{article}

\usepackage{graphicx}
\usepackage{amsmath}
\usepackage{amsfonts}
\usepackage{color}
\usepackage{latexsym}
\usepackage{tikz}
\usepackage{pgfplots}
\pgfplotsset{compat=1.18} 
\usepackage{tablefootnote}
\usepackage{setspace}
\usepackage{epsfig}
\usepackage{multirow}
\usepackage{float}
\usepackage{array}
\usepackage{xcolor}
\usepackage{amssymb,amsthm}
\usepackage{xurl}
\usepackage{hyperref}    
\usepackage{cleveref}    

\definecolor{antiquefuchsia}{rgb}{0.57, 0.36, 0.51}
\definecolor{MyViolet}{rgb}{0.45,0.08,0.95}
\definecolor{MyBrown}{rgb}{0.45,0.08,0}
\definecolor{MyDarkBlue}{rgb}{0,0.08,0.45}

\usepackage{makecell,booktabs}
\usepackage[version=4]{mhchem}   
\usepackage[strict]{changepage}

\usepackage{amsmath}
\usepackage{amsfonts}
\usepackage{latexsym}
\usepackage{amssymb}

\newtheorem{thm}{Theorem}[section]
\newtheorem{theorem}[thm]{Theorem}

\newtheorem{lemma}[thm]{Lemma}

\newtheorem{corollary}[thm]{Corollary}

\newtheorem{remark}[thm]{Remark}

\newcommand{\beq}{\begin{equation}}
\newcommand{\eeq}{\end{equation}}
\newcommand{\beqa}{\begin{eqnarray}}
\newcommand{\eeqa}{\end{eqnarray}}
\newcommand{\beqas}{\begin{eqnarray*}}
\newcommand{\eeqas}{\end{eqnarray*}}
\newcommand{\bi}{\begin{itemize}}
\newcommand{\ei}{\end{itemize}}

\usepackage{algorithm}
\usepackage{algpseudocode}

\begin{document}

\title{The method of ellipcenters with momentum and relaxation for convex quadratic minimization}

\date{}
 
	 \maketitle

\begin{center}
\begin{tabular}{ccc}
\begin{tabular}{c}
Roger Behling\\
Department of Mathematics, UFSC\\
Blumenau, Santa Catarina, Brazil\\
{\tt rogerbehling@gmail.com}
\end{tabular}&
&
\begin{tabular}{c}
Vincent Guigues\\
School of Applied Mathematics, FGV\\
Praia de Botafogo, Rio de Janeiro, Brazil\\
{\tt vincent.guigues@fgv.br}
\end{tabular}
\end{tabular}
\end{center}

\begin{center}
\begin{tabular}{c}
Harry Oviedo\\
Department of Mathematics\\
Instituto Tecnológico Autónomo de México, ITAM
\\
Mexico city, Mexico\\
{\tt harry.oviedo@uai.cl}
\end{tabular}
\end{center}

\par {\textbf{Abstract.}} The method of ellipcenters (ME) is a recent technique developed for unconstrained minimization. Its iteration relies only on first order information and consists of building a suitable two-dimensional ellipse that tries to capture intrinsic ill-conditioning of the problem. The center of the ellipse is taken as the next iterate, which justifies the name of the method. ME was already shown to converge linearly when the objective function is smooth and strongly convex. The special case when the objective is  quadratic was studied in the first paper on ME and is also the subject of our work here. 
In this paper, we propose two variants of ME:
RelaxME which adds in ME an additional relaxation step
at each iteration and MomME which embeds momentum
in ME. We prove linear convergence of RelaxME and MomME and,
in particular, we show convergence of RelaxME and
MomME (and also of ME) in one iteration when
the matrix of the quadratic form has only two distinct
eigenvalues. Finally, we provide the results
of numerical experiments which compare on five
optimization problems (and different combinations of parameters defining these problems) ME, RelaxME, and MomME,
with 3 other optimizers: the conjugate gradient method \cite{hestenes1952methods}, Barzilai and Borwein gradient method with long step \cite{barzilai1988two}, and the gradient method with adaptive spectral step length \cite{Frassoldati2008}.
On most instances, MomME provides the smallest number of
iterations and conjugate gradient and MomME provide
the smallest CPU times and similar CPU times. This opens optimistic possibilities for ME with momentum in broader settings.\\




\maketitle

\section{Introduction}

We present two procedures (relaxation and momentum) in order to accelerate the Method of Ellipcenters (ME), a recent first order technique that employs centers of suitable two-dimensional ellipses for minimizing a strongly convex and differentiable function $f: \mathbb{R}^n \rightarrow \mathbb{R}$. ME iterates by moving from a current iterate $x_k\in\mathbb{R}^n$ to a point $x_{k+1}\in\mathbb{R}^n$ that is the center of a special ellipse $E_k$.

Let us start by describing ME formally as it was introduced in \cite{behling2026method}. Consider that $x_k\in \mathbb{R}^n$ is the current iterate of ME. If it is stationary, that is, $\nabla f(x_k)=0$, we are already at the unique minimizer of $f$ and ME stops with success. When $x_k$ is non-stationary, ME computes the unique point $y_k\neq x_k$ such that $y_k=x_k-t_k\nabla f(x_k)$ and $f(y_k)=f(x_k)$. Then, if $\mbox{span}\{\nabla f(x_k),\nabla f(y_k)\}$ is one-dimensional, ME returns $(x_k+y_k)/2$. If $\nabla f(x_k)$ and $\nabla f(y_k)$ are linearly independent, ME takes the next iterate as the center of an ellipse $E_k$ fulfilling the following three properties:
\begin{description}
   \item[ME1] $E_k$ is an ellipse contained in \begin{equation}\label{defpik}
\Pi_k:=\Big\{x\in \mathbb{R}^n :x=x_k+\mbox{span}\{\nabla f(x_k),\nabla f(y_k)\}\Big\};
\end{equation}
   \item[ME2] $E_k$ is orthogonal to $\nabla f(x_k)$ at $x_k$;
   \item[ME3] $E_k$ is orthogonal to $\nabla f(y_k)$ at $y_k$.
\end{description}
The intention of ME is to capture ill-conditioning through two-dimensional elliptical interpolations. Its geometry is illustraded in Figure \ref{Grafico}.

\begin{figure}[h]
    \centering
    \includegraphics[width=0.7\textwidth]{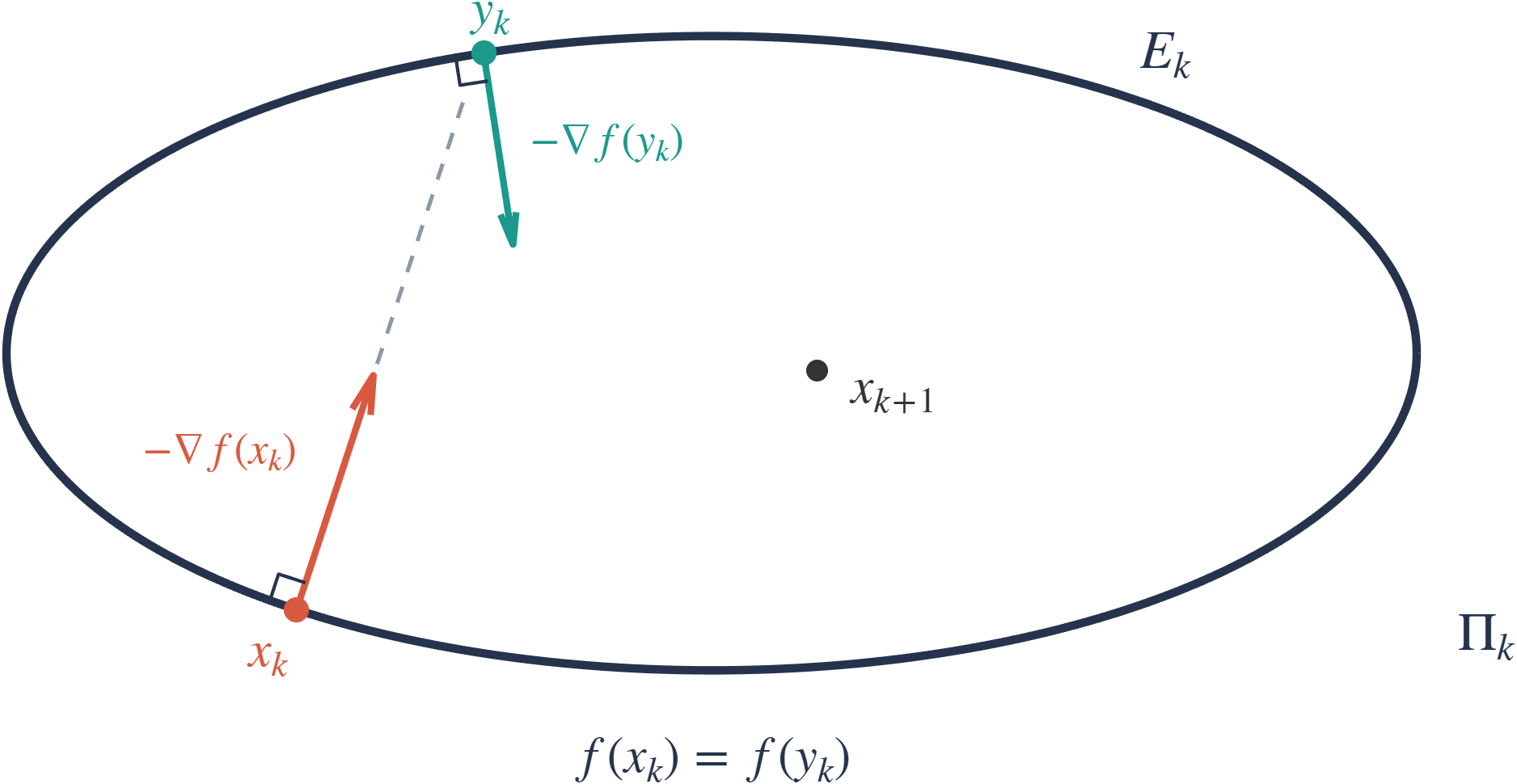}
    \caption{Illustration of an iteration of ME.}
    \label{Grafico}
\end{figure}

The construction of an ellipse, \textit{i.e.}, a two-dimensional ellipsoid is not in vain. There are several reasons for that. One is that the zigzagging of the Cauchy method (steepest descent with optimal line-search) for a strictly convex quadratic objective gets, in the limit, smashed into a two-dimensional affine space related to the two extreme eigenvalues of the Hessian (see 
\cite{akaike1959}, \cite{nocedal02}, and \cite{GonzagaKaras2013}). Moreover, level sets of strictly convex quadratics are ellipsoids which, in turn, provide good approximations of level sets of a wide range of smooth convex functions. So ME is somewhat designed to encage the minimizer by ellipsoids constrained to the two-dimensional affine spaces $\Pi_k$ described above. For a strictly convex quadratic in $\mathbb{R}^2$, ME fences in the minimizer perfectly and yields finite convergence in one single step. In particular, it coincides with Newton's method in this simple context.

Very recently, ME has been studied for strongly convex minimization. In \cite{behling2026ME_Strong},  the set of ellipses and associated centers has been characterized using standard Analytic Geometry. It turns out that there is an infinite number of ellipses satisfying ME1, ME2, and ME3. Furthermore, the correspondent centers lie in an appropriate half-line starting at $\frac{1}{2}(x_k+y_k)$ and going in a certain direction $d_k$, that depends 
on $x_k, y_k$ and their gradients with respect to $f$ (see the formula in \cite{behling2026ME_Strong}). The starting point $\frac{1}{2}(x_k+y_k)$ is somewhat degenerate, as the associated "ellipse" becomes simply the segment with end points $x_k$ and $y_k$. One can then, among all ellipcenters, look for the point $x_{k+1}$ that minimizes $f$ along this half-line $S_k:=\{x\in\mathbb{R}^n:x=\frac{1}{2} (x_k+y_k)+td_k, \mbox{ }t\geq 0\}$ of centers of ellipses. This particular center of ellipse given by the minimizer of $f$ in $S_k$ will be referred to as the best ellipcenter with respect to $x_k$ and denoted by $\mbox{ME}(x_k)$. It was shown in \cite{behling2026ME_Strong} that, by defining ME in this way, the method achieves linear convergence to the unique global minimizer of the smooth strongly convex objective $f$. 

As said earlier, this work will focus on $f$ being not only strongly convex and differentiable, but also quadratic. That said, it will not be necessary for us to go further into details on the construction of the half-line of ellipcenters $S_k$. We can avoid this discussion, because the best ellipcenter $\mbox{ME}(x_k)$ happens to be the minimizar of $f$ in the plane $\Pi_k$ when $f$ is quadratic, see \cite{behling2026method}, and thus $\mbox{ME}(x_k)$ is essentially provided by the outcome of a simple $2\times 2$ linear system of equations. We observe that this characterization of $\mbox{ME}(x_k)$ being the minimizer of $f$ in $\Pi_k$ is no longer true for general smooth strongly convex minimization, see the discussion in \cite{behling2026ME_Strong} and how it implies that the method in \cite{YunierME2026} is different from ME.

For the rest of the paper, we consider
optimization methods to solve the problem
\begin{equation}\label{formfqd}
\min_{x \in \mathbb{R}^n} f(x) := \frac{1}{2} x^T A x - b^T x+c
\end{equation}
where is $f$ a strongly convex quadratic function, i.e.,
\( A \in \mathbb{R}^{n \times n} \) is a symmetric positive definite  matrix, \( b \in \mathbb{R}^n \), and $c$ is a scalar.
We will denote by $x_*$ the optimal solution
of \eqref{formfqd}.

We now list the three main contributions of our paper. The first one regards the improvement on the rate of convergence of ME itself. In \cite{behling2026method}, it was shown that ME, for minimizing the quadratic $f$, has a convergence rate better or equal to $\eta$, where $\eta$ is the linear rate for the Cauchy method. We show that, actually, the rate of ME is at least as good as the rate provided by two successive Cauchy steps, that is, ME converges with rate $\eta^2$. Our second contribution consists of considering in ME the general relaxation framework by \cite{RaydanSvaiter}. We call RelaxME the corresponding variant of ME.
We prove that RelaxME also enjoys linear convergence and shares the same theoretical linear rate $\eta^2$ of ME for the case when the relaxation parameter is $1$. In our experiments, however, RelaxME consistently outperforms ME. Finally, we employ Nesterov's momentum \cite{nesterov1983} within ME, getting what we call MomME. MomME still maintains the rate $\eta^2$ in theory, but behaves way faster than both ME and RelaxME, in practice. 
Moreover, we prove
finite convergence of ME, RelaxME, and
MomME when $A$
has at most two different eigenvalues.

Before moving on with the presentation of the organization of the paper, let us briefly point out some facts that may already help with the understanding of the sections ahead. We note that $f$ being a strictly convex quadratic gives us some nice properties. Indeed, it turns out that the Cauchy point computed at the non-stationary iterate $x_k$, denoted here by $\mbox{Cauchy}(x_k)$ is simply given by $\frac{1}{2}(x_k+y_k)$, with $y_k$ being the auxiliary point defined in ME. More than that, $\nabla f(\mbox{Cauchy}(x_k))=A(\frac{1}{2}(x_k+y_k))-b=\frac{1}{2}\nabla f(x_k)+\frac{1}{2}\nabla f(y_k)$. This means, in particular, that $\nabla f(\mbox{Cauchy}(x_k))\in \mbox{span}\{\nabla f(x_k),\nabla f(y_k)\}$. Also, one can prove that if $\nabla f(x_k)$ and $\nabla f(y_k)$ are linearly dependent, $\nabla f(\mbox{Cauchy}(x_k))=0$ and ME (and also RelaxME and MomME) is done with the minimization of $f$, see Lemma \ref{remterminates} below. Finally, when $\mbox{span}\{\nabla f(x_k),\nabla f(y_k)\}$ is two-dimensional, it follows trivially from our observations that $\mbox{Cauchy}(\mbox{Cauchy}(x_k))\in \Pi_k$. However, recall that $x_{k+1}:=\mbox{ME}(x_k)$ is the minimizer of $f$ in $\Pi_k$ now that $f$ is quadratic. Long story short, we are saying that one iteration of ME is better or equal, by means of function value, than two consecutive steps of the Cauchy method. This is the main reason for us to get the rate $\eta^2$ mentioned above.

\if{
The idea here in principle is to pick a better point than the one given by the classical gradient method, also known as the steepest descent method or Cauchy method \cite{Cauchy1847}. More than that, the intention of $E_k$ is to emulate the level curve of $f$ restricted to $\Pi_k$ and encage a minimizer, which is slightly related to the recent
Circumcentered-Reflection-Method (CRM) proposed and studied in
for instance  \cite{Behling2024b,crm1}. Figure \ref{Grafico} illustrates an iteration of ME.

It is easy to see that even if $f$ has a global minimizer, the ME iteration might not be well-defined, there could simply not exist a point $y^k$ as required in the procedure. If, however, $f$ is strongly convex or coercive  an ellipse as described above can be constructed, and right away from ME1, ME2, and ME3 one would get that $x^{k+1}-x^k$ is a descent direction for $f$ at $x^k$. Nevertheless, there could be more ellipses $E_k$ fulfilling ME1-ME3 and, among them, a specific one would have to be chosen. Thus, in order to facilitate the understanding of ME and to clearly present its philosophy, we will assume along the remaining of the paper that $f$ is a strongly convex quadratic function. More general cases will be treated in future work. The strongly convex quadratic case alone is interesting and elegant. More than being able to build a convenient ellipse enjoying the properties ME1, ME2 and ME3, we get closed formulas for $y^k$ and can easily compute $x^{k+1}$ analytically when $f$ is strongly  convex and quadratic.

Accelerating the gradient method with first-order tools has long been a popular subject of research in the field of Continuous Optimization, see for instance \cite{grimmer2023optimal,lan2015bundle,mishchenko2020adaptive,nesterov2015universal,renegar2022simple} and \cite{zhou2024adabb}. Perhaps one of the most famous of these methods is given in the paper \cite{nesterov1983} by Y. Nesterov (fine tuning
of this algorithm is discussed in \cite{gonzagakaras}). He embedded inertia in the gradient method and was able to derive an iteration with best possible complexity for a first-order method, \textit{i.e.}, a method using first derivatives only.
Another popular and efficient accelerated gradient method for composite
convex minimization is FISTA which was introduced in \cite{fista}.
Another first-order algorithm that enjoys optimal complexity in theory is the Ellipsoid method \cite{Khachiyan1979}. Nevertheless, in practice, the Ellipsoid method, which is different from what we are proposing here, is not attractive.

}\fi

Our paper is organized as follows.
In Section \ref{sec:relaxme}, we describe
RelaxME, show several properties of this
method, and show its convergence.
In Section \ref{sec:3}, we describe
the variant MomME of ME and prove its
convergence. In Section \ref{sec:finiteconv},
we prove the convergence of ME, RelaxME,
and MomME in one iteration when
$A$ has at most two different eigenvalues.
Finally, in Section \ref{sec:num}, we provide the results
of numerical experiments which compare on five
optimization problems (and different combinations of parameters defining these problems) ME, RelaxME, MomME,
with 3 other optimizers: the conjugate gradient method \cite{hestenes1952methods}, Barzilai and Borwein gradient method with long step \cite{barzilai1988two}, and the gradient method with adaptive spectral step length \cite{Frassoldati2008}.
In most instances, MomME provides the smallest number of
iterations and conjugate gradient and MomME provide
the smallest CPU times and similar CPU times. This opens optimistic possibilities for ME with momentum in broader settings.

\section*{Notation and terminology}

For $x \in \mathbb{R}^n$, and $A$ an $n \times n$ definite positive
matrix, we denote by $\|\cdot\|_A$ the norm defined by
$$
\|x\|_A=\sqrt{x^T A x}
$$
induced by the scalar product $
\langle x,y\rangle_A=x^T A y
$
for $x,y \in \mathbb{R}^n$.
Given
an $n \times n$ definite positive
matrix $A$, we say that vectors
$x,y \in \mathbb{R}^n$ are 
$A$-conjugate if 
$\langle x,y \rangle_A=0$.

For vectors $x,y \in \mathbb{R}^n$, we will use the notation
$\langle x,y \rangle$ for the scalar product $x^T y$ with corresponding norm
$\|\cdot\|$.

\section{RelaxME}\label{sec:relaxme}

\subsection{Motivation and algorithm}

In this subsection, we introduce and analyze the convergence of the relaxed
variant RelaxME of ME.
More precisely, at
a given iteration
where $x_k$ was the last
point computed by
RelaxME (the point computed in the previous iteration), we first compute $\tilde x_{k+1}$
performing one iteration
of ME. Given a relaxation parameter 
$\theta \in (0,1]$,
we then compute the next iterate $x_{k+1}$
as the convex combination
$x_{k+1}=(1-\theta)x_{k} + \theta \tilde x_{k+1}$
between the previous iterate $x_{k}$ and
$\tilde x_{k+1}$. 
We recall that given
$x_{k}$ such that 
$\nabla f(x_k) \neq 0$, an iteration of ME first computes 
$y_k=x_k-t_k \nabla f(x_k)$ such that
$f(y_k)=f(x_k)$.
If $\nabla f(y_k)$ and
$\nabla f(x_k)$ are linearly independent,
then $\tilde x_{k+1}$
minimizes $f$ in the affine space 
$x_k+\mbox{Span}(\nabla f(x_k), \nabla f(y_k))$. It is worth mentioning that
the minimizer $x_k+\alpha_k \nabla f(x_k) + \beta_k \nabla f(y_k)$
of $f$ in this affine space is known
in analytic form. If $\nabla f(y_k)$ and
$\nabla f(x_k)$ are linearly dependent,
then $\tilde x_{k+1}=(1/2)(x_k+y_k)$.
The statement of
Relaxed ME is given in 
Algorithm \ref{Alg1}.
As shown in \cite{behling2026method}, the formula given for $t_k$ in this algorithm
ensures that $f(y_k)=f(x_k)$. Moreover,
as shown in Lemma \ref{lemgk}, the sequences
$(g_k)$ and $(r_k)$ 
computed iteratively
satisfy 
$g_k=\nabla f(x_k)$ and $r_k=\nabla f(y_k)$
for all $k \geq 0$.

\begin{algorithm}
\caption{Relaxed Method of Ellipcenters for Convex Quadratic
Minimization (RelaxME)}
\label{Alg1}

\begin{algorithmic}[1]
\Statex \textbf{Initialization.}
\State Let $A \in \mathbb{R}^{n \times n}$ be symmetric positive definite,
$b \in \mathbb{R}^{n}$, and $x_0 \in \mathbb{R}^{n}$.
\State Set
\[
\widetilde{x}_0 = x_0,\qquad
g_0 = \nabla f(x_0),\qquad
\theta \in (0,1],\qquad
k = 0.
\]

\While{$\lVert g_k \rVert \geq \varepsilon$}

    \State $w_k = A g_k$

    \State $\displaystyle
        t_k = \frac{2\lVert g_k\rVert^2}
        {g_k^{\top}w_k}$

    \State $y_k = x_k - t_k g_k$

    \State $r_k = g_k - t_k w_k$

    \If{$g_k$ and $r_k$ are linearly independent}

        \State $\displaystyle
        (\alpha_k,\beta_k)
        =
        \operatorname*{arg\,min}_{\alpha,\beta}
        f\bigl(x_k+\alpha g_k+\beta r_k\bigr)$

        \State $\widetilde{x}_{k+1}
        = x_k+\alpha_k g_k+\beta_k r_k$

    \Else

        \State $\displaystyle
        \widetilde{x}_{k+1}
        = \frac{1}{2}(x_k+y_k)$

        \State $x_{k+1}=
        \widetilde{x}_{k+1}$
\State $g_{k+1}
        = \nabla f(\widetilde{x}_{k+1})$

        \State {\textbf{break}}

    \EndIf

    \If{$k=0$}

        \State $x_{k+1} = \widetilde{x}_{k+1}$

        \State $g_{k+1}
        = \nabla f(\widetilde{x}_{k+1})$

    \Else
    \State $x_{k+1}
    = (1-\theta)x_k+\theta\widetilde{x}_{k+1}$

    \State $g_{k+1}
    = (1-\theta)g_k
    +\theta\nabla f(\widetilde{x}_{k+1})$

    \EndIf

    \State $k = k+1$

\EndWhile
\end{algorithmic}
\end{algorithm}

In the next subsection, we prove several properties of RelaxME.

\subsection{Properties of RelaxME}

In the following lemma, we show that for every $k \geq 0$, $g_k$
is the gradient of $f$ at $x_k$ and
$r_k$ is the gradient of $f$ at $y_k$.
\begin{lemma}\label{lemgk} For every $k \geq 0$, we
have $g_k=\nabla f(x_k)$ and $r_k=\nabla f(y_k)$.
\end{lemma}
\begin{proof} The proof is by induction on $k$.
For $k=0$, the initialization step of the algorithm gives $g_0=\nabla f(x_0)$
and for $k=1$ the algorithm computes
$g_1=\nabla f(x_1)$.
Now assume that for $k \geq 1$ we have
$g_k=\nabla f(x_k)$. We want to show that
$g_{k+1}=\nabla f(x_{k+1})$.
At iteration $k+1$, if the algorithm
stops at the ${\textbf{break}}$
command at line 15, we have $g_{k+1}=\nabla f(x_{k+1})$
by definition of $g_{k+1}$.
Otherwise, we have
${x}_{k+1} = (1-\theta)x_k+\theta {\tilde x}_{k+1}$, which implies
\begin{eqnarray}
\nabla f(x_{k+1})&= & Ax_{k+1}-b\nonumber\\
&= & A\Big((1-\theta)x_k+\theta {\tilde x}_{k+1}\Big)-b\nonumber\\
&= & (1-\theta)(Ax_{k} - b)+ \theta(A \tilde x_{k+1}-b)\nonumber\\
&= & (1-\theta) \nabla f(x_k) + \theta(A \tilde x_{k+1}-b)\nonumber\\
&= & (1-\theta)g_k+ \theta \nabla f(\tilde x_{k+1})\;\;\mbox{ (by the induction assumption)}\nonumber\\
&= & g_{k+1}\;\;\mbox{(by definition of $g_{k+1}$)}\label{pr21}
\end{eqnarray}
which shows the relation $g_k=\nabla f(x_k)$ for all $k \geq 0$. Finally,
$$
r_k=g_k-t_kw_k=\nabla f(x_k)-t_k A g_k =
Ax_k-b-t_k A g_k=Ay_k-b=\nabla f(y_k),
$$
which achieves the proof of the lemma.
\end{proof}

Observe that the condition
$g_k=\nabla f(x_k)$ and 
$r_k=\nabla f(y_k)$
linearly dependent can
be written
$\|\nabla f(x_k)\|=\|\nabla f(y_k)\|$.
Indeed, it is easy to check
that 
$\langle \nabla f(x_k),\nabla f(y_k) \rangle = -\|\nabla f(x_k)\|^2$
and vectors $g_k$ and $r_k$
are linearly dependent
if and only if 
$|\cos(\angle (\nabla f(x_k), \nabla f(y_k)))|=1$ which can be written equivalently
$$
1=\left| 
\frac{\langle \nabla f(x_k),\nabla f(y_k) \rangle}{\|\nabla f(x_k)\|\|\nabla f(y_k)\|}
\right|=\frac{\|\nabla f(x_k)\|}{\|\nabla f(y_k)\|}.
$$

We now introduce error terms
\begin{equation}\label{defeketk}
e_k=x_k-x_* \mbox{ and }\tilde e_k=\tilde x_k-x_*
\end{equation}
that will be useful in our analysis as well as 
\begin{equation}\label{defdk}
d_k=e_k-\tilde e_{k+1}.
\end{equation}

Another interesting property of RelaxME, shown
in the following lemma, is that if at a given iteration $k$,
vectors $g_k$ and $r_k$ are linearly dependent then this is the last iteration
of the method which outputs the optimal solution $x_{*}$. This justifies the 
{\textbf{break}} command at line 15 of the algorithm.

\begin{lemma}\label{remterminates}
If $g_k$ and $r_k$ are linearly dependent
we have $x_{k+1}=\tilde x_{k+1}=x_*$, i.e.,
RelaxME has computed the minimum of $f$ in a finite number of
iterations.
\end{lemma}
\begin{proof}
If $g_k$ and $r_k$ are linearly dependent,
there is $\lambda$ such that $Ag_k=\lambda g_k$. Therefore $\lambda>0$ is eigenvalue and $g_k$ eigenvector
of $A$. It follows that
$$
t_k=\frac{2\|g_k\|^2}{g_k^T A g_k}=\frac{2}{\lambda}
$$
and from the iterations of the method, we have
\begin{equation}\label{ekfc0}
\tilde x_{k+1}=\frac{1}{2}(x_k+y_k)=x_k-\frac{t_k}{2}g_k=x_k-\frac{1}{\lambda}g_k.
\end{equation}
Since $Ax_{*}=b$, we then have 
\begin{equation}\label{Aek}
A e_k  \stackrel{\eqref{defeketk}}{=}A(x_k-x_*)=Ax_k-b=\nabla f(x_k)=g_k
\end{equation}
where the last equality comes from Lemma \ref{lemgk}. Therefore
we have 
\begin{equation}\label{ekfc1}
e_k=A^{-1}g_k=\frac{1}{\lambda}g_k
\end{equation}
and 
\begin{equation}
x_{k+1} = \tilde x_{k+1} \stackrel{\eqref{ekfc0},\eqref{ekfc1}}{=} x_k-e_k\stackrel{\eqref{defeketk}}{=}x_*.
\end{equation}
\end{proof}

The following lemma provides
additional insights on RelaxME which are
useful for its convergence analysis.

\begin{lemma}
The following holds 
for error terms $(e_k)$, $(\tilde e_k)$, and
sequence $(d_k)$: for every $k \geq 0$ such that
$g_k$ and $r_k$ are linearly independent,
\begin{itemize}
\item[(A)] $d_k$ and $\tilde e_{k+1}$
are $A$-conjugate;
\item[(B)] we have
\begin{equation}\label{ekortpitl}
\|e_k\|_A^2 = \|\tilde e_{k+1}\|_A^2 +\|d_k\|_A^2;
\end{equation}
\item[(C)] we have
\begin{equation}\label{crucekp1l}
\|e_{k+1}\|_A^2 = \|\tilde e_{k+1}\|_A^2 + (1-\theta)^2 \|d_{k}\|_A^2. 
\end{equation}
\end{itemize}
\end{lemma}
\begin{proof} (A) Observe that if $g_k$ and $r_k$ are linearly independent,
\begin{equation}\label{span1}
d_k=x_k-\tilde x_{k+1}
\in \mbox{Span}(g_k,r_k).
\end{equation}
Next, by the optimality conditions for the problem
solved in the linearly independent case,
$$
\langle \nabla f(\tilde x_{k+1}),g_k\rangle = 0  \mbox{ and }\langle \nabla f(\tilde x_{k+1}),r_k\rangle=0
$$
which implies
\begin{equation}\label{span2}
\nabla f(\tilde x_{k+1})
\in \mbox{Span}(g_k,r_k)^{\perp}.
\end{equation}
Combining \eqref{span1} and \eqref{span2}, we obtain
\begin{equation}\label{span3}
0=d_k^T \nabla f(\tilde x_{k+1}) =
d_k^T[A \tilde x_{k+1} -b]=d_k^T A \tilde e_{k+1}=
\langle d_k, \tilde e_{k+1} \rangle_A, 
\end{equation}
which achieves the proof of (A).

\par (B) We deduce from (A) that
\begin{eqnarray}
\|e_k\|_A^2 & = & \|\tilde e_{k+1}+d_k\|_A^2 \nonumber \\ 
&= & \|\tilde e_{k+1}\|_A^2 +\|d_k\|_A^2 + 2\langle d_k, \tilde e_{k+1} \rangle_A \nonumber\\
&\stackrel{\eqref{span3}}{=} & \|\tilde e_{k+1}\|_A^2 +\|d_k\|_A^2,\label{ekortpit}
\end{eqnarray}
which proves \eqref{ekortpitl}.

\par (C) Next, we can write
\begin{eqnarray}
e_{k+1} & = & x_{k+1}-x_* \nonumber \\ 
&= & (1-\theta)x_k + \theta \tilde x_{k+1} -x_* \nonumber\\
&= & (1-\theta)(x_k-x_*) + \theta (\tilde x_{k+1} -x_*) \nonumber\\
&= & (1-\theta)e_k + \theta \tilde e_{k+1} \nonumber\\
&= & (1-\theta)({\tilde e}_{k+1}+d_k) + \theta \tilde e_{k+1} \nonumber\\
&= & \tilde e_{k+1}+(1-\theta)d_k.\label{ekortpit2}
\end{eqnarray}
Using \eqref{span3}, \eqref{ekortpit2}, and
the Pythagorean theorem, we have
\begin{equation}\label{crucekp1}
\|e_{k+1}\|_A^2 = \|\tilde e_{k+1}\|_A^2 + (1-\theta)^2 \|d_{k}\|_A^2, 
\end{equation}
which achieves the proof of \eqref{crucekp1l}.
\end{proof}

\par In what follows we denote by 
$0< \lambda_1 \leq \lambda_2 \leq \ldots \leq \lambda_n$ the eigenvalues of matrix $A$.

The following lemma
will be useful for the convergence analysis of RelaxME, shown in Section \ref{sec:convergence}.

\begin{lemma}\label{lemineqcruc1}
For every $k \geq 0$, we have
\begin{equation}\label{ineqfriststep}
f(\tilde x_{k+1})-f(x_*) \leq \rho^4 (f(x_k)-f(x_*))
\end{equation}
where 
\begin{equation}\label{defrho}
\rho = \frac{\kappa-1}{\kappa+1} 
\mbox{ with }\kappa=\frac{\lambda_n}{\lambda_1}.
\end{equation}

\end{lemma}
\begin{proof} We consider two cases:
the case where $g_k$ and $r_k$
are linearly independent and the case where they are
linearly dependent.\\

\par {\textbf{Case where $g_k$ and $r_k$ are linearly independent.}}\footnote{The reader can find a second proof of
\eqref{ineqfriststep} in this case in the Appendix.}

In this case, $\tilde x_{k+1} \in x_k+ \mbox{Span}(g_k,r_k)$.
Since  $r_k=g_k-t_kAg_k$ with $t_k>0$, we have
$\mbox{Span}(g_k,r_k)=\mbox{Span}(g_k,Ag_k)$.
Therefore, $\tilde x_{k+1}$ minimizes $f$ on the affine 2-dimensional Krylov space
$x_k+\mbox{Span}(g_k,Ag_k)$ and for a point
$x_k(\alpha,\beta)=x_k + \alpha g_k + \beta Ag_k$ in this affine space the error
is
\begin{eqnarray*}
x_k + \alpha g_k + \beta Ag_k-x_* & \stackrel{\eqref{Aek}}{=} &
e_k+\alpha Ae_k+ \beta A^2 e_k=P(A) e_k
\end{eqnarray*}
where $P$ is the polynomial of degree 2 given by
$$
P(\lambda)=1+\alpha \lambda+ \beta \lambda^2
$$
satisfying $P(0)=1$.
For any polynomial $P$
of degree 2
with $P(0)=1$, we therefore have
\begin{eqnarray}
f(\tilde x_{k+1})-f(x_*) &= & \frac{1}{2} \|\tilde x_{k+1}-x_*\|_A^2\nonumber \\
& \leq & \frac{1}{2}\|P(A)e_k\|_A^2 \label{anyphave}.
\end{eqnarray}
Now consider the polynomial of degree 2
$$
P(\lambda)=\left(1-\frac{2 \lambda}{\lambda_1+\lambda_n}\right)^2
$$
that satisfies $P(0)=1$. Observe that for every 
$\lambda \in [\lambda_1,\lambda_n]$, we have
$$
\left \lvert 1- \frac{2\lambda}{\lambda_{1}+\lambda_{n}} \right \rvert \leq \rho.
$$
It follows that 
\begin{equation}\label{ineqpoly}
\lvert P(\lambda)\rvert \leq \rho^2\;\;\forall \lambda \in [\lambda_1,\lambda_n].
\end{equation}
Let $(v_i)_{i=1}^n$ be an orthonormal basis of eigenvectors
of $A$ with $Av_i=\lambda_i v_i$. Writing 
$e_k=\sum_{i=1}^n a_i v_i$, we have
$$
P(A)e_k=\sum_{i=1}^n a_i P(\lambda_i) v_i
$$
which gives
\begin{equation}\label{bsupP}
\|P(A) e_k\|_A^2 = \sum_{i=1}^n \lambda_i a_i^2 P(\lambda_i)^2.
\end{equation}
Using \eqref{ineqpoly}, we obtain
$\lvert P(\lambda_i) \rvert \leq \rho^2$ for all $i=1,\ldots,n,$ which, plugged
into \eqref{bsupP}, yields
\begin{equation}\label{pfinal}
\|P(A)e_k\|_A^2 \leq \rho^4 \sum_{i=1}^n \lambda_i a_i^2 = \rho^4 \|
e_k\|_A^2.
\end{equation}
Combining \eqref{anyphave} and \eqref{pfinal}, we obtain
\begin{equation}\label{crucxtilde}
f(\tilde x_{k+1})-f(x_*) \leq \frac{1}{2} \rho^4 \|e_k\|_A^2 =\rho^4(f(x_k)-f(x_*)),
\end{equation}
which achieves the proof of \eqref{ineqfriststep}
in the linearly independent case.\\

\par {\textbf{Case where $g_k$ and $r_k$ are linearly dependent.}}
In this case, using Lemma \ref{remterminates}, we have
$\tilde x_{k+1} = x_*$
 and 
\eqref{ineqfriststep} holds trivially since the left-hand side
$f(\tilde x_{k+1})-f(x_*)$ is equal to zero.
\end{proof}

\subsection{Convergence Analysis}\label{sec:convergence}

The following theorem shows that RelaxME converges linearly to
$x_*$, the minimum of $f$, at a rate a least equal to 
\begin{equation}\label{linrate}
\bar \eta:=(1-\theta) + \theta\left( \frac{\kappa-1}{\kappa+1}\right)^4 \in [0, 1)
\end{equation}
where $\kappa=\frac{\lambda_n}{\lambda_1}$ is the condition number of matrix $A$.

\begin{theorem}\label{theorem1}\label{theo1}
Let $x_0 \in \mathbb{R}^n$ be any initial point for RelaxME. Let $0 < \lambda_1 \leq \lambda_2 \leq \dots \leq \lambda_n$ be the eigenvalues of $A$ and let
$\kappa$
be the condition number of $A$. Then, RelaxME generates a sequence $\{x_k\} \subset \mathbb{R}^n$ that converges to $x_{*}$, the unique minimizer of the quadratic $f$. Moreover, the convergence is linear and there is
$
\eta \leq \bar \eta
$
where $\bar \eta$ is given by 
\eqref{linrate}
such that for all $k \geq 0$, we have
\begin{equation}\label{eq1_teo1}
    f(x_k) - f(x_*) \leq \eta^k (f(x_0) - f(x^*))
 \end{equation}
\textit{and}
\begin{equation}\label{eq2_teo1}
    \|x_k - x_*\|_A \leq \sqrt{\eta^{k}} \|x_0 - x_*\|_A.
\end{equation}
\end{theorem}
\begin{proof}
Using the update for $x_{k+1}$ in RelaxME, the convexity of $f$, and \eqref{ineqfriststep}, we have
\begin{eqnarray}
f(x_{k+1})-f(x_*) & \leq  & (1-\theta)(f(x_k)-f(x_*)) + \theta(f(\tilde x_{k+1})-f(x_*)) \nonumber\\
& \stackrel{\eqref{ineqfriststep}}{\leq} & ((1-\theta)+\theta \rho^4)(f(x_k)-f(x_*)) \nonumber \\
& = & {\bar \eta} (f(x_k)-f(x_*)) \nonumber
\end{eqnarray}
where $\rho$ is given by \eqref{defrho}, which achieves the proof of \eqref{eq1_teo1}.

Next, we have
\begin{eqnarray}
f(x_k)-f(x_*)& = & \frac{1}{2}x_k^T A x_k - b^T x_k - \frac{1}{2}x_*^T A x_* + b^T x_* \nonumber\\
& = & \frac{1}{2}x_k^T A x_k - x_*^T A x_k + \frac{1}{2}x_*^T A x_*\nonumber\\
& = & \frac{1}{2} \|e_k\|_A^2 =  \frac{1}{2} \|x_k-x_*\|_A^2 \label{pr22} \\
& \leq &  {\bar \eta}^k (f(x_0)-f(x_*))  \nonumber\\
& = & \frac{1}{2} {\bar \eta}^k \|x_0-x_*\|_A^2  \nonumber
\end{eqnarray}
which achieves the proof of \eqref{eq2_teo1}.
\end{proof}

\par We now provide an iteration complexity result for the relaxed method of ellipcenters
which is an immediate 
Corollary of 
Theorem \ref{theo1}. It provides an explicit expression for the number of iterations
$K$
for
$f(x_k) - f(x_{*})$ to be below some desired tolerance $\epsilon$.
\begin{corollary}
Let $(x_k)$ be the sequence generated by RelaxME and $\epsilon>0$ be a given positive real number. Then $f(x_K) - f(x_{*}) \leq \epsilon$,  for every $K$ such that
$$ K \geq \frac{1}{\ln(1/\bar \eta)} \emph{ln}\,\left( \frac{f(x_{0}) - f(x_{*})}{\epsilon} \right). $$
Moreover,  $\|x_N - x_{*}\|_A \leq \epsilon $, for every $N$ satisfying
$$ N \geq \frac{2}{\ln(1/\bar \eta)} \emph{ln}\,\left( \frac{\|x_0 - x_{*}\|_A}{\epsilon} \right). $$
\end{corollary}

Next, we show that the convergence
rate can be improved for RelaxME.
More precisely, we have the following:

\begin{theorem}\label{theo2}
Let $x_0 \in \mathbb{R}^n$ be any initial point for RelaxME. Let $0 < \lambda_1 \leq \lambda_2 \leq \dots \leq \lambda_n$ be the eigenvalues of $A$ and let
$\kappa$
be the condition number of $A$. Then, RelaxME converges
linearly to $x_{*}$, the unique minimizer of the quadratic $f$: there is
$
\eta \leq \eta^*
$
where 
$$
\eta^*:= 1-\theta + 
\theta \frac{(\lambda_n-\lambda_1)^4}{(\lambda_n^2+6 \lambda_1 \lambda_n + \lambda_1^2)^2}=
1-\theta+\theta
\frac{(\kappa-1)^4}{(\kappa^2+6\kappa+1)^2}
$$
such that for all $k \geq 0$ inequalities
\eqref{eq1_teo1}
and \eqref{eq2_teo1} hold.
\end{theorem}

Rate $\eta^*$ given by Theorem
\ref{theo2} is better than
 rate $\bar \eta$  given by Theorem
\ref{theo1} 
for $\lambda_1<\lambda_n$ since
in this case we have
$\eta^*<\bar \eta$.\\

\par {\textbf{Proof of Theorem \ref{theo2}.}} It is easy to see
that it suffices to show
that \eqref{ineqfriststep} holds
with $\rho^4$ replaced by $\tilde \rho^4$
where 
$$
\tilde\rho^2 = \frac{(\lambda_n-\lambda_1)^2}{\lambda_n^2+6 \lambda_1 \lambda_n+\lambda_1^2}.
$$
Following the proof of Lemma \ref{lemineqcruc1}, it suffices to consider
the case where $g_k$ and $r_k$
are linearly independent. Let $T_2$ be the Chebyshev
polynomial of degree 2 given
by $T_2(x)=2x^2-1$ and define
the polynomial of degree 2
$$
\displaystyle P(\lambda)=\frac{T_2\left(\frac{2\lambda-(\lambda_n+\lambda_1)}{\lambda_n-\lambda_1}\right)}{T_2\left(-\frac{\lambda_n+\lambda_1}{\lambda_n-\lambda_1}\right)}.
$$
Using the fact that 
for all
$\lambda \in [\lambda_1,\lambda_n]$,
we have 
$$
\frac{2\lambda-(\lambda_n+\lambda_1)}{\lambda_n-\lambda_1} \in [-1,1],
$$
and $\lvert T_2(x) \rvert \leq 1$ for all
$x \in [-1,1]$, we deduce
that  
\begin{equation}\label{ineqp2}
\lvert P(\lambda) \rvert \leq
\tilde \rho^2 =\frac{1}{\left \lvert {T_2\left(-\frac{\lambda_n+\lambda_1}{\lambda_n-\lambda_1}\right)} \right \rvert} = \frac{(\lambda_n-\lambda_1)^2}{\lambda_n^2+6 \lambda_1 \lambda_n+\lambda_1^2}
\end{equation}
for all $\lambda \in [\lambda_1,\lambda_n]$.

Since $P$ is a polynomial of degree 2 with $P(0)=1$, we can reproduce the proof  of Theorem
\ref{theo1} replacing inequality 
\eqref{ineqpoly} by this
inequality \eqref{ineqp2}
to obtain
\begin{equation}\label{crucxtilde2}
f(\tilde x_{k+1})-f(x_*) \leq \frac{1}{2} {\tilde \rho}^4 \|e_k\|_A^2 = 
{\tilde \rho}^4 (f(x_k)-f(x_*)).
\end{equation}
This shows that inequalities
\eqref{eq1_teo1}
and \eqref{eq2_teo1} hold
for every  $\eta \leq \eta^*$.
$\hfill \square$\\

\par We can still improve this rate, as shown in the following theorem.

\begin{theorem}\label{theo3}
Let $x_0 \in \mathbb{R}^n$ be any initial point for RelaxME. Let $0 < \lambda_1 \leq \lambda_2 \leq \dots \leq \lambda_n$ be the eigenvalues of $A$ and let
$\kappa$
be the condition number of $A$. Then, RelaxME converges
linearly to $x_{*}$, the unique minimizer of the quadratic $f$: there is
$
\eta \leq \hat \eta
$
where 
$$
\hat \eta=
(1-\theta)^2 + \theta (2-\theta)\frac{(\kappa-1)^4}{(\kappa^2+6\kappa+1)^2}
$$
such that for all $k \geq 0$ inequalities
\eqref{eq1_teo1}
and \eqref{eq2_teo1} hold.
Moreover, if $\lambda_1<\lambda_n$
then 
we have $\hat \eta<\eta^*$.
\end{theorem}

\par {\textbf{Proof of Theorem \ref{theo3}.}}
We have already seen
that if $g_k$ and
$r_k$ are linearly dependent, then this is the last iteration of the algorithm which has computed in a finite number of iterations the optimal solution. It suffices to consider the case when
for all iterations $k$, 
vectors $g_k$ and $r_k$
are linearly independent.
In this case, it follows from \eqref{ekortpit}
and \eqref{crucekp1} that
\begin{eqnarray}
f(x_{k+1})-f(x_*)
= \frac{1}{2}\|e_{k+1}\|_A^2 & = &
\frac{1}{2}\Big((1-\theta)^2 \|e_k\|_A^2 + 
\theta (2-\theta)\|\tilde e_{k+1}\|_A^2\Big) \nonumber \\
& \stackrel{\eqref{crucxtilde2}}{\leq} &
\frac{1}{2}\Big(\Big[ (1-\theta)^2 + \theta (2-\theta){\tilde \rho}^4 \Big] \|e_{k}\|_A^2 \Big)\nonumber \\
& = &
\Big[ (1-\theta)^2 + \theta (2-\theta){\tilde \rho}^4 \Big] (f(x_k)-f(x_*)). \nonumber
\end{eqnarray}
We have therefore shown that
\eqref{eq1_teo1} and \eqref{eq2_teo1}
also hold with $\bar \eta$
replaced by $\hat \eta$
and if $\lambda_1<\lambda_n$
then $\hat \eta<\eta^*$.
$\hfill \square$

\section{The method of ellipcenters with momentum}
\label{sec:3}

\subsection{Algorithm}

This section introduces a second acceleration strategy for the 
method of
ellipcenters, which incorporates inertial terms. Specifically, we develop a two--step update scheme that first computes an intermediate point, $\tilde x_{k+1}$, using a standard iteration of the method of ellipcenters (Algorithm \ref{Alg1} with $\theta = 1$). This point is then improved by incorporating momentum, defining $x_{k+1}$ as a weighted sum of between $\tilde x_{k+1}$ and $x_{k-1}$
\begin{equation}
x_{k+1} = (1-\mu_k)\tilde x_{k+1} + \mu_k x_{k-1}, \label{eq1_sec3}
\end{equation}
where the weight $\mu_k\in\mathbb{R}$ is optimally determined as follows
\begin{equation}
\mu_k = \textrm{arg}\min_{\mu\in\mathbb{R}} f(\, (1-\mu)\tilde x_{k+1} + \mu x_{k-1}  \,).\label{eq2_sec3}
\end{equation}

The method of ellipcenters with momentum 
is given in Algorithm \ref{Alg2}.

\begin{algorithm}
\caption{Ellipcenter Method with Momentum for Convex Quadratic
Minimization (MomME)}
\label{Alg2}

\begin{algorithmic}[1]

\Require $A \in \mathbb{R}^{n \times n}$ symmetric positive definite,
$b \in \mathbb{R}^{n}$, $x_0 \in \mathbb{R}^{n}$,
and tolerance $\varepsilon > 0$

\State $\widetilde{x}_0 = x_0$
\State $g_0 = \nabla f(x_0)$
\State $k = 0$

\While{$\lVert g_k \rVert > \varepsilon$}

    \State $w_k = A g_k$

    \State $\displaystyle
        t_k =
        \frac{2\lVert g_k\rVert^2}
             {g_k^{\top}w_k}$

    \State $y_k = x_k-t_k g_k$

    \State $r_k = g_k-t_k w_k$

    \If{$g_k$ and $r_k$ are linearly independent}

        \State $\displaystyle
        (\alpha_k,\beta_k)
        =
        \operatorname*{arg\,min}_{(\alpha,\beta)\in\mathbb{R}^2}
        f\bigl(x_k+\alpha g_k+\beta r_k\bigr)$

        \State $\widetilde{x}_{k+1}
        = x_k+\alpha_k g_k+\beta_k r_k$

    \Else

        \State $\displaystyle
        \widetilde{x}_{k+1}
        = \frac{1}{2}(x_k+y_k)$

 \State $x_{k+1} = \widetilde{x}_{k+1}$

        \State $g_{k+1}
        = \nabla f(\widetilde{x}_{k+1})$

        \State {\textbf{break}}

    \EndIf

    \If{$k=0$}

        \State $x_{k+1} = \widetilde{x}_{k+1}$

        \State $g_{k+1}
        = \nabla f(\widetilde{x}_{k+1})$

    \Else

        \State $\widetilde{s}_{k-1}
        = \widetilde{x}_{k+1}-x_{k-1}$

        \State $\widetilde{v}_{k-1}
        = \nabla f(\widetilde{x}_{k+1})-g_{k-1}$

        \State $\displaystyle
        \mu_k
        =
        \frac{
            \nabla f(\widetilde{x}_{k+1})^{\top}
            \widetilde{s}_{k-1}
        }{
            \widetilde{s}_{k-1}^{\top}
            \widetilde{v}_{k-1}
        }$

        \State $x_{k+1}
        =
        (1-\mu_k)\widetilde{x}_{k+1}
        +\mu_k x_{k-1}$

        \State $g_{k+1}
        =
        (1-\mu_k)\nabla f(\widetilde{x}_{k+1})
        +\mu_k g_{k-1}$

    \EndIf

    \State $k = k+1$

\EndWhile

\end{algorithmic}
\end{algorithm}

Same as Algorithm \ref{Alg1}, we have
that $g_k=\nabla f(x_k)$, which can be shown
easily by induction on $k$. Indeed, the relation
holds for $k=0$ and $k=1$. Now assume that 
$g_k=\nabla f(x_k)$ for some $k \geq 1$.
Then 
\begin{eqnarray}
\nabla f(x_{k+1})&= & Ax_{k+1}-b\nonumber\\
&= & A\Big( \mu_k x_{k-1}+(1-\mu_k) {\tilde x}_{k+1}\Big)-b\nonumber\\
&= & (1-\mu_k) \nabla f(\tilde x_{k+1}) + \mu_k \nabla f(x_{k-1})\nonumber\\
&= & (1-\mu_k) \nabla f(\tilde x_{k+1}) + \mu_k g_{k-1}\;\;\mbox{ (by the induction assumption)}\nonumber\\
&= & g_{k+1}\;\;\mbox{(by definition of $g_{k+1}$)}\label{pr212}
\end{eqnarray}
which achieves the proof that $g_k=\nabla f(x_k)$
for all $k \geq 0$ by induction. As in Algorithm 
\ref{Alg1}, this implies
that $r_k=\nabla f(y_k)$.
This shows that, same as in 
Algorithm \ref{Alg1}, at every iteration, given $x_k$, Algorithm \ref{Alg2} computes
$\tilde x_{k+1}$ using an iteration of the method of ellipcenters, as announced before. Next, we check that
the computations
of $\tilde s_{k-1}$,
$\tilde v_{k-1}$, and 
$\mu_k$ amount to choosing
$\mu_k$ that solves problem
\eqref{eq2_sec3}, as announced before.
Indeed, the optimality conditions
for problem \eqref{eq2_sec3} that 
$\mu_k$ solves is
\begin{eqnarray}
0 & = & \langle x_{k-1}-\tilde x_{k+1},
\nabla f( \tilde x_{k+1} + \mu_k(x_{k-1}-{\tilde x}_{k+1})) 
\rangle \nonumber \\
 & = & \langle x_{k-1}-\tilde x_{k+1},
\nabla f( \tilde x_{k+1} )  + \mu_k(\nabla f(x_{k-1})-\nabla f({\tilde x}_{k+1})) \rangle \nonumber\\ 
 & = & \langle x_{k-1}-\tilde x_{k+1},
\nabla f( \tilde x_{k+1} )  + \mu_k(g_{k-1}-\nabla f({\tilde x}_{k+1}))
\rangle, \nonumber 
\end{eqnarray}
which gives
\begin{equation}
\mu_k = \frac{\nabla f(\tilde x_{k+1})^{\top}(\tilde x_{k+1}-x_{k-1})}{(\tilde x_{k+1}-x_{k-1})^{\top}(\nabla f(\tilde x_{k+1})-g_{k-1})}\label{eq3_sec3}
\end{equation}
and which indeed corresponds to the formula used in 
Algorithm \ref{Alg2}
for $\mu_k$. We have therefore justified
all the computations in Algorithm \ref{Alg2}.

Observe also that in the linearly independent case, the auxiliary point $\tilde x_{k+1}$ can be re--written as
\begin{equation}
\tilde x_{k+1} = x_k + (\alpha_k + \beta_k)g_k - \beta_kt_k\,Ag_k. \label{eq4_sec3}
\end{equation}
Combining \eqref{eq4_sec3} with \eqref{eq1_sec3} we obtain
$$ x_{k+1} = x_k + (1-\mu_k)(\alpha_k+\beta_k)g_k - (1-\mu_k)\beta_kt_k\,Ag_k - \mu_ks_{k-1}, $$
where $s_{k-1} = x_k - x_{k-1}$ is the commonly called momentum term. Thus, $x_{k+1}$ belongs to three-dimensional affine space $x_k + \textrm{span}\{ g_k,Ag_k,s_{k-1} \}$. Note that this differs from the conjugate gradient method which computes $x_{k+1}$ as the unique minimizer of $f(\cdot)$ over the bi-dimensional space $x_k + \textrm{span}\{ g_k,s_{k-1} \}$.

\subsection{Convergence of MomME}

By the minimization property \eqref{eq2_sec3}, we directly have
$$f(x_{k+1}) - f(x_{*}) \leq f(\tilde x_{k+1}) - f(x_{*}).$$

This shows that the method of ellipcenters with momentum is at least as good as the method of ellipcenters and therefore inherits the convergence rate of Algorithm \ref{Alg1} with $\theta = 1$. Hence, we can state the following convergence result.

\begin{theorem}\label{theorem2}
Let $x_0 \in \mathbb{R}^n$ be any initial point. Let $0 < \lambda_1 \leq \lambda_2 \leq \dots \leq \lambda_n$ be the eigenvalues of $A$. Then, Algorithm \ref{Alg2} generates a sequence $(x_k) \subset \mathbb{R}^n$ that converges to $x_{*}$, the unique minimizer of the quadratic $f$. Moreover, the convergence is linear and we have $\eta \leq \frac{(\kappa-1)^4}{(\kappa^2+6\kappa+1)^2}$ such that for all $k \geq 0$,
\begin{equation}\label{eq1_teo2}
    f(x_k) - f(x_*) \leq \eta^k (f(x_0) - f(x^*))
\end{equation}
\textit{and}
\begin{equation}\label{eq2_teo2}
    \|x_k - x_*\|_A \leq \sqrt{\eta^{k}} \|x_0 - x_*\|_A.
\end{equation}
\end{theorem}

\section{Finite convergence 
of ME, RelaxME, and MomME
in one single iteration}\label{sec:finiteconv}

In this section, we prove
Theorem \ref{finiteconv} below
which states 
that ME, RelaxME, and MomME converge in at most one iteration when
the matrix $A$ has either one single eigenvalue or has only two distinct eigenvalues.

\begin{theorem}\label{finiteconv}
Sequences $(x_k)$ generated by 
ME, RelaxME, and MomME converge to the minimizer 
$x_*$ of $f$
in at most one iteration when
the matrix $A$ has either one single eigenvalue or has only two distinct eigenvalues.
\end{theorem}
\begin{proof} 
First, observe that all first iterations of the three algorithms
ME, RelaxME, and MomME are the same.
If $g_0=0$ then
$x_0=x_*$ and no iteration is performed. Assume now that 
$g_0 \neq 0$. 
By
\eqref{defeketk},
\begin{equation}\label{g0ae0}
g_0=Ax_0-b=Ae_0 \mbox{ or equivalently }
e_0=A^{-1}g_0.
\end{equation}
Moreover, 
$r_0=g_0-t_0 A g_0$
for $t_0>0$ which gives
\begin{equation}\label{spgrq}
\mbox{span}(g_0,r_0)=
\mbox{span}(g_0,Ag_0).
\end{equation}

We now consider two cases: the case when $A$ has a single eigenvalue $\lambda>0$ and the case when
$A$ has only two distinct eigenvalues $\lambda_1, \lambda_2>0$.

\par {\textbf{Case when $A$ has a single eigenvalue $\lambda>0$.}} In this case, $A=\lambda I$,  
$\mathbb{R}^n=\mbox{Ker}(A-\lambda I)$
and $A g_0 = \lambda g_0$.
By \eqref{g0ae0}, we then have 
$e_0=(1/\lambda)g_0$ which implies 
$$
t_0=\frac{2\|g_0\|^2}{g_0^T A g_0}= \frac{2}{\lambda}
$$
and it follows that 
$$
r_0=g_0-\frac{2}{\lambda}A g_0=-g_0. 
$$
We are therefore in the situation where
$r_0$ and $g_0$ are linearly dependent
and we can follow the proof
of Lemma \ref{remterminates}
to show that $\tilde x_1=x_1$. Indeed, using $y_0=x_0-t_0g_0$, we have 
\begin{eqnarray}
\tilde x_1 & = & \frac{1}{2}(x_0+y_0) =  x_0 - \frac{t_0}{2}g_0
 =  x_0 - \frac{1}{\lambda}g_0\nonumber\\
& = & x_0 -e_0=x_*. 
\end{eqnarray}
Since at the first iteration, the algorithm sets $x_1=\tilde x_1$, it outputs in one iteration the optimal solution $x_*$.\\

\par {\textbf{Case when $A$ has only two distinct eigenvalues $\lambda_1>0$
and $\lambda_2>0$.}} Since $A$ is symmetric, by the spectral theorem,
there is an orthogonal matrix $U=[u_1,u_2,\ldots,u_n]$ whose columns
$u_1,u_2,\ldots,u_n$,
are 
eigenvectors of $A$ and a diagonal matrix $\Lambda$ whose positive entries in the diagonal are the eigenvalues of $A$ such that $A=U \Lambda U^T$ with
$U U^T=I$. Then 
\begin{equation}\label{specth}
(A-\lambda_1 I)(A-\lambda_2 I)=
U (\Lambda-\lambda_1 I) U^T U (\Lambda-\lambda_2 I) U^T =U (\Lambda-\lambda_1 I)(\Lambda-\lambda_2 I) U^T. 
\end{equation}
It is easy to check that 
$(\Lambda-\lambda_1 I)(\Lambda-\lambda_2 I)=0$, which, combined with 
\eqref{specth}, gives
\begin{equation}\label{specth1}
(A-\lambda_1 I)(A-\lambda_2 I)=0. 
\end{equation}
Relation \eqref{specth1} can be written
equivalently 
$$
A^2-(\lambda_1+\lambda_2)A+\lambda_1 \lambda_2 I =0
$$
which gives
\begin{equation}\label{specth2}
A^{-1}=\frac{\lambda_1 + \lambda_2}{\lambda_1 \lambda_2}I-\frac{1}{\lambda_1 \lambda_2}A. 
\end{equation}
Using \eqref{specth2}, we obtain
\begin{equation}\label{specth3}
e_0 = A^{-1} g_0 =\frac{\lambda_1 + \lambda_2}{\lambda_1 \lambda_2}g_0-\frac{1}{\lambda_1 \lambda_2}Ag_0. 
\end{equation}
Therefore, 
\begin{equation}\label{specth4}
e_0  \in \mbox{span}(e_0,Ag_0)=
\mbox{span}(g_0,r_0).
\end{equation}
It follows that 
$$
x_*=x_0-e_0 \in x_0 + \mbox{span}(g_0,r_0).
$$
Using the relation
$$
\mathbb{R}^n=\mbox{Ker}(A-\lambda_1 I) \oplus \mbox{Ker}(A-\lambda_2 I),
$$
we obtain 
$$
g_0= g_0^{(1)} + g_0^{(2)} 
$$
for $g_0^{(1)},$ $g_0^{(2)}$ orthogonal
with $g_0^{(1)} \in \mbox{Ker}(A-\lambda_1 I)$
and $g_0^{(2)} \in \mbox{Ker}(A-\lambda_2 I)$.

If $g_0^{(1)}$ or $g_0^{(2)}$ is null, we 
can reproduce the computations of the
first case where $A$ has a single eigenvalue
to show that $\tilde x_1=x_*$ and the algorithm
terminates in one iteration.
Indeed, in this case, we can assume without loss of generality that 
$g_0^{(2)}=0$
and
$g_0=g_0^{(1)}$
which gives
$A g_0=\lambda_1 g_0$.
By \eqref{g0ae0}, we then have 
$e_0=(1/\lambda_1)g_0$ which implies 
$$
t_0=\frac{2\|g_0\|^2}{g_0^T A g_0}= \frac{2}{\lambda_1}
$$
and it follows that 
$$
r_0=g_0-\frac{2}{\lambda_1}A g_0=-g_0. 
$$
We are therefore in the situation where
$r_0$ and $g_0$ are linearly dependent
and we can follow the proof
of Lemma \ref{remterminates}
to show that $x_1=\tilde x_1=x_1$.

Otherwise, if both $g_0^{(1)}$ and $g_0^{(2)}$
are nonnull, then $g_0$ and $Ag_0$
are linearly independent. To see this, 
we compute $Ag_0 = Ag_0^{(1)} + Ag_0^{(2)}=
\lambda_1 g_0^{(1)} + \lambda_2 g_0^{(2)}$.
By contradiction, if $A g_0 = \alpha g_0$ then
$$
(\lambda_1-\alpha)g_0^{(1)} + (\lambda_2-\alpha)g_0^{(2)}=0. 
$$
Since vectors $g_0^{(1)}$ and $g_0^{(2)}$
are orthogonal,  they are linearly independent
which implies $\alpha=\lambda_1=\lambda_2$,
a contradiction. 
We have therefore shown that $g_0$ and $Ag_0$
are linearly independent which implies, by \eqref{spgrq},
that $g_0$ and $r_0$ are linearly independent.
In this case, the algorithm computes the minimizer $\tilde x_1$ of 
$f$ in the affine space
$x_0+\mbox{span}(g_0,r_0)$.
Since we have shown
that $x_*$, the unique global minimizer of $f$, belongs to this affine space, we have
that $\tilde x_1=x_*$,
$x_1=\tilde x_1=x_*$
and again the algorithm terminates in one iteration.

\end{proof}

\section{Numerical Experiments}\label{sec:num}

In this section, we illustrate the effectiveness of our proposed algorithms on a variety of convex quadratic problems. We implemented RelaxME and MomME in Julia setting $\theta = 0.9$ in all experiments. In order to compare our proposal with other methodologies existing in the literature, we include comparisons with the following methods: the conjugate gradient method (CG) \cite{hestenes1952methods}, Barzilai and Borwein gradient method with long step (BB1) \cite{barzilai1988two}, the gradient method with adaptive spectral step length (ABBmin1) with $\tau=0.8$ and $m = 9$, proposed in \cite{Frassoldati2008} and the original method of ellipcenters (ME) \cite{behling2026method}.  All experiments were performed with an AMD Ryzen 5 2600 processor, with access to 24GB of RAM and 3.4GHz.
The Julia code of all 
implementations corresponding to these experiments can be found on
github at 
{\url{https://github.com/vguigues/MomME}}.\\

\subsection{Experiment 1: random convex quadratic problems}
Our first experiment assesses the performance of the algorithms when applied to the large-scale dense problems considered in \cite{Friedlander1999}. In particular, the matrix $A$ is randomly generated as follows: $A = PDP^{\top}\in\mathbb{R}^{n \times n}$ where $P$ is the product of three Householder matrices,
$$ P = (I - 2v_1v_1^{\top})(I - 2v_2v_2^{\top})(I - 2v_3v_3^{\top}) $$
where $v_1$, $v_2$ and $v_3$ are three random vectors belonging to the unitary sphere, and $$D = \textrm{diag}(d_1,d_2,\ldots,d_n)$$ is a diagonal matrix whose $i$--th diagonal entry is given by
$$ d_i = \textrm{exp}\left(\frac{i-1}{n-1}\mbox{ncond}\right), \quad \forall i\in\{1,2,\ldots, n\}, $$
where \textrm{ncond} is a positive integer. Note that $A$ becomes increasingly ill--conditioned as \textrm{ncond} increases. The vector $b\in\mathbb{R}^{n}$ is assembled as $b=Ax_{*}$, where $x_{*} = 2\textrm{rand}(n,1)-\textrm{ones}(n,1)$ using Matlab  notation and the initial point is the null vector. We let all algorithms run up to $K=50000$ iterations and stop them at iteration $k<K$ if $\|g_k\|<1e$-$7$. In this experiment, we consider three values of $n$ ($n \in \{1000, 2500, 5000\}$) and four values of $\text{ncond}$ ($\text{ncond} \in \{3, 6, 9, 12\}$). For each pair $(n, \text{ncond})$, we evaluate the algorithms across five independent problems. We report the average number of iterations (Iter), CPU time in seconds (Time), and final gradient norm (NrmG) for all methods. Table \ref{tab:1} collects the numerical results associated with this experiment. It can be seen from this table
that our method MomME requires much less iterations
that all competing optimizers on all instances.
For 
$n=2500$ and
ncond=$6$, CG, MomME, and
ABBmin1 are the quickest
optimizers and provide close mean CPU times.
For the remaining 11 combinations of the pair
$(n,\mbox{ncond})$, CG and MomME are the quickest and provide
very close mean CPU times.

\begin{table}
\centering
\caption{Numerical results corresponding to Experiment 1.}
\label{tab:1}
\begin{tabular}{c c c c c c c}
\hline		
Method & Iter & Time & NrmG & Iter & Time & NrmG  \\
\hline
& \multicolumn{3}{c}{ $(n,\emph{ncond}) = (1000,3)$ } & \multicolumn{3}{c}{ $(n,\emph{ncond}) = (1000,6)$ } \\
\hline
ME	& 48.7 & 0.0254 & $<$e-16 & 845.7 & 0.4172 & $<$e-16 \\
RelaxME	& 34.0 & 0.0154 & $<$e-16 & 184.0 & 0.0701 & $<$e-16 \\
CG	& 45.3 & 0.0103 & $<$e-16 & 206.3 & 0.0371 & $<$e-16 \\
BB1	& 62.7 & 0.0163 & $<$e-16 & 259.0 & 0.0581 & $<$e-16 \\
ABBmin1	& 58.0 & 0.0119 & $<$e-16 & 250.0 & 0.0424 & $<$e-16 \\
MomME	& 23.0 & 0.0111 & $<$e-16 & 105.7 & 0.0407 & $<$e-16 \\
\hline
& \multicolumn{3}{c}{ $(n,\emph{ncond}) = (1000,9)$ } & \multicolumn{3}{c}{ $(n,\emph{ncond}) = (1000,12)$ } \\
\hline
ME	& 16712.3 & 9.7304 & $<$e-16 & 329493.7 & 140.8681 & $<$e-16 \\
RelaxME	& 874.0 & 0.5789 & $<$e-16 & 3901.7 & 1.5546 & $<$e-16 \\
CG	& 887.0 & 0.2324 & $<$e-16 & 3686.3 & 0.4768 & $<$e-16 \\
BB1	& 1448.7 & 0.3648 & $<$e-16 & 8498.3 & 1.7435 & $<$e-16 \\
ABBmin1	& 1141.7 & 0.3426 & $<$e-16 & 5015.3 & 0.6843 & $<$e-16 \\
MomME	& 484.7 & 0.1953 & $<$e-16 & 2254.7 & 0.6829 & $<$e-16 \\
\hline
& \multicolumn{3}{c}{ $(n,\emph{ncond}) = (2500,3)$ } & \multicolumn{3}{c}{ $(n,\emph{ncond}) = (2500,6)$ } \\
\hline
ME	& 49.0 & 0.5899 & $<$e-16 & 871.7 & 10.4034 & $<$e-16 \\
RelaxME	& 35.3 & 0.2910 & $<$e-16 & 165.0 & 1.4051 & $<$e-16 \\
CG	& 46.0 & 0.1825 & $<$e-16 & 215.3 & 0.9570 & $<$e-16 \\
BB1	& 56.7 & 0.4171 & $<$e-16 & 302.0 & 2.1332 & $<$e-16 \\
ABBmin1	& 59.0 & 0.2318 & $<$e-16 & 246.0 & 0.9561 & $<$e-16 \\
MomME	& 23.0 & 0.1830 & $<$e-16 & 108.0 & 1.0381 & $<$e-16 \\
\hline
& \multicolumn{3}{c}{ $(n,\emph{ncond}) = (2500,9)$ } & \multicolumn{3}{c}{ $(n,\emph{ncond}) = (2500,12)$ } \\
\hline
ME	& 17013.7 & 207.8736 & $<$e-16 & 341493.0 & 3445.9426 & $<$e-16 \\
RelaxME	& 925.7 & 8.5134 & $<$e-16 & 4318.7 & 29.9695 & $<$e-16 \\
CG	& 964.3 & 4.3331 & $<$e-16 & 4270.7 & 14.4453 & $<$e-16 \\
BB1	& 1597.7 & 12.8237 & $<$e-16 & 9026.0 & 58.8927 & $<$e-16 \\
ABBmin1	& 1117.7 & 4.7207 & $<$e-16 & 5081.7 & 19.0859 & $<$e-16 \\
MomME	& 496.7 & 4.4994 & $<$e-16 & 2281.7 & 15.9311 & $<$e-16 \\
\hline
& \multicolumn{3}{c}{ $(n,\emph{ncond}) = (5000,3)$ } & \multicolumn{3}{c}{ $(n,\emph{ncond}) = (5000,6)$ } \\
\hline
ME	& 50.0 & 1.7424 & $<$e-16 & 884.3 & 31.0129 & $<$e-16 \\
RelaxME	& 33.7 & 0.7976 & $<$e-16 & 179.7 & 4.2374 & $<$e-16 \\
CG	& 47.0 & 0.5710 & $<$e-16 & 220.0 & 2.7046 & $<$e-16 \\
BB1	& 59.0 & 1.3261 & $<$e-16 & 285.0 & 6.4330 & $<$e-16 \\
ABBmin1	& 58.7 & 0.6993 & $<$e-16 & 250.7 & 2.9492 & $<$e-16 \\
MomME	& 23.0 & 0.5638 & $<$e-16 & 109.3 & 2.7710 & $<$e-16 \\
\hline
& \multicolumn{3}{c}{ $(n,\emph{ncond}) = (5000,9)$ } & \multicolumn{3}{c}{ $(n,\emph{ncond}) = (5000,12)$ } \\
\hline
ME	& 17370.3 & 592.4319 & $<$e-16 & 349393.7 & 13728.9943 & $<$e-16 \\
RelaxME	& 795.0 & 18.5366 & $<$e-16 & 4193.7 & 124.5071 & $<$e-16 \\
CG	& 1000.3 & 11.3145 & $<$e-16 & 4510.3 & 65.1129 & $<$e-16 \\
BB1	& 1275.7 & 28.5448 & $<$e-16 & 7518.0 & 202.3553 & $<$e-16 \\
ABBmin1	& 1181.3 & 13.6488 & $<$e-16 & 5369.3 & 74.6858 & $<$e-16 \\
MomME	& 505.0 & 11.4702 & $<$e-16 & 2321.3 & 66.0011 & $<$e-16 \\
\hline
\end{tabular}
\end{table}

\subsection{Experiment 2: two point boundary value problems}
In the second group of experiments, we evaluate the performance of the considered procedures using a practical test problem. Specifically, the linear system involves a tridiagonal matrix $A = [a_{ij}] \in \mathbb{R}^{n \times n}$ defined by
$$a_{i,i} = \frac{2}{h^2}, \quad a_{i,i-1} = -\frac{1}{h^2} \,\, (i \neq 1), \quad a_{i,i+1} = -\frac{1}{h^2} \,\, (i \neq n),$$
for all $i \in \{1, 2, \dots, n\}$, where $h = 11/n$. This class of linear systems frequently arises when numerically solving two--point boundary--value problems, see \cite{Meyer,Oviedo2020}. For this experiment, the initial point $x_0$ is set to the null vector in $\mathbb{R}^n$ and the 
entries of vector $b$ are randomly generated in the interval $[-1, 1]$. For the stopping criterion, we used a maximum number of iterations of $K = 50000$ and a tolerance for the gradient norm of $\epsilon = 1e$-$4$. We test all the algorithms across different dimensions ($n = 100, 200, 300, 400, 500, 600$). Table \ref{tab:2} summarizes the average number of iterations, average CPU time (in seconds), and the average final gradient norm computed over five independent runs for each dimension. 
On this experiment, all optimizers are quick
but CG and our method MomME provide the smallest number
of iterations and are the two
quickest on 5 of the 6 instances.

\begin{table}
\centering
\caption{Numerical results corresponding to Experiment 2.}
\label{tab:2}
\begin{tabular}{c c c c c c c}
\hline		
Method & Iter & Time & NrmG & Iter & Time & NrmG  \\
\hline
& \multicolumn{3}{c}{ $n=100$ } & \multicolumn{3}{c}{ $n=200$ } \\
\hline
ME	& 7907.6 & 0.0067 & $<$e-16 & 32116.2 & 0.0463 & $<$e-16 \\
RelaxME	& 608.6 & 0.0097 & $<$e-16 & 1086.4 & 0.0810 & $<$e-16 \\
CG	& 100.0 & 0.0002 & $<$e-16 & 200.0 & 0.0007 & $<$e-16 \\
BB1	& 861.4 & 0.0013 & $<$e-16 & 1997.8 & 0.0057 & $<$e-16 \\
ABBmin1	& 464.2 & 0.0006 & $<$e-16 & 964.8 & 0.0017 & $<$e-16 \\
MomME	& 117.4 & 0.0006 & $<$e-16 & 209.6 & 0.0024 & $<$e-16 \\
\hline
& \multicolumn{3}{c}{ $n=300$ } & \multicolumn{3}{c}{ $n=400$ } \\
\hline
ME	& 70926.4 & 0.1603 & $<$e-16 & 122748.8 & 0.3327 & $<$e-16 \\
RelaxME	& 1670.2 & 0.1233 & $<$e-16 & 2063.2 & 0.1052 & $<$e-16 \\
CG	& 300.0 & 0.0017 & $<$e-16 & 400.0 & 0.0030 & $<$e-16 \\
BB1	& 2643.4 & 0.0160 & $<$e-16 & 3826.8 & 0.0233 & $<$e-16 \\
ABBmin1	& 1377.4 & 0.0077 & $<$e-16 & 1774.8 & 0.0160 & $<$e-16 \\
MomME	& 306.0 & 0.0066 & $<$e-16 & 402.8 & 0.0086 & $<$e-16 \\
\hline
& \multicolumn{3}{c}{ $n=500$ } & \multicolumn{3}{c}{ $n=600$ } \\
\hline
ME	& 187710.6 & 0.6338 & $<$e-16 & 278392.8 & 1.0913 & $<$e-16 \\
RelaxME	& 2333.0 & 0.0529 & $<$e-16 & 3433.0 & 0.0835 & $<$e-16 \\
CG	& 500.0 & 0.0038 & $<$e-16 & 600.0 & 0.0056 & $<$e-16 \\
BB1	& 4593.4 & 0.0298 & $<$e-16 & 6490.4 & 0.0474 & $<$e-16 \\
ABBmin1	& 2304.8 & 0.1006 & $<$e-16 & 2942.0 & 0.0951 & $<$e-16 \\
MomME	& 507.0 & 0.0139 & $<$e-16 & 609.0 & 0.0165 & $<$e-16 \\
\hline
\end{tabular}
\end{table}

\subsection{Experiment 3: sparse quadratic problems with real--data}
For the third experiment, we evaluated 28 strictly convex quadratic test functions using large--scale sparse matrices $A \in \mathbb{R}^{n \times n}$ ($n \geq 1000$) taken from the UF Sparse Matrix Collection \cite{Davis}\footnote{The matrices used, belonging to the UF Sparse Matrix Collection, can be downloaded from \url{https://sparse.tamu.edu/}}. The vector $b$ was generated via $b = Ax_*$, assuming that the unique solution is given by $x_* = (1, 2, 3, \dots, n)^\top$, while the initial point was selected as $x_0 = (1, -1, 1, -1, \ldots, 1)^\top$, in other words, the $i$--th entry of the vector $x_0$ is equal to $-1$ if $i$ is even and is $1$ otherwise. The algorithms were executed for up to $K = 50,000$ iterations and stopped at iteration $k \leq K$ if the relative residual satisfied $\|\nabla f(x_k)\|_2 < \epsilon $ with $\epsilon = 10^{-7}$. Comparative results are detailed in Tables \ref{tab:31}, \ref{tab:32},  and \ref{tab:33}. In these tables, ``NrmG'' denotes the last gradient norm, ``'Iter' denotes the number of iterations carried out by the method and ``Time'' denotes the total computational time in seconds. Remarkably, on two
instances, ME, RelaxME, and MomME only require
1 iteration. Among the 28
optimization problems, MomME was
not competitive only for the two
instances mhd3200b and mhd4800b
but for instance mhd3200b RelaxME
and CG are the quickest, providing similar CPU times while for instance mhd4800b RelaxME
is by far the quickest method. 
On most of the remaining 26 instances, MomME provides the smallest
number of iterations and MomME and CG provide
the smallest CPU times. 

\begin{table}
\centering
\caption{Numerical results corresponding to Experiment 3.}
\label{tab:31}
\begin{tabular}{c c c c c c c}
\hline		
Method & Iter & Time & NrmG & Iter & Time & NrmG  \\
\hline
& \multicolumn{3}{c}{ bcsstm02 } & \multicolumn{3}{c}{ bcsstm05 } \\
\hline
ME	& 19 & 5.9300e-5 & $<$e-16 & 27 & 0.0003170 & $<$e-16 \\
RelaxME	& 14 & 5.2200e-5 & $<$e-16 & 20 & 0.0002062 & $<$e-16 \\
CG	& 12 & 1.7600e-5 & $<$e-16 & 18 & 8.9700e-5 & $<$e-16 \\
BB1	& 26 & 5.1200e-5 & $<$e-16 & 32 & 0.0002476 & $<$e-16 \\
ABBmin1	& 23 & 2.8100e-5 & $<$e-16 & 35 & 0.0001560 & $<$e-16 \\
MomME	& 9 & 3.1000e-5 & $<$e-16 & 13 & 0.0001420 & $<$e-16 \\
\hline
& \multicolumn{3}{c}{ bcsstm07 } & \multicolumn{3}{c}{ bcsstm09 } \\
\hline
ME	& 12524 & 1.0564755 & $<$e-16 & 1 & 0.0003488 & $<$e-16 \\
RelaxME	& 596 & 0.0438500 & $<$e-16 & 1 & 0.0002017 & $<$e-16 \\
CG	& 369 & 0.0124541 & $<$e-16 & 2 & 0.0004870 & $<$e-16 \\
BB1	& 1808 & 0.0936485 & $<$e-16 & 4 & 0.0009644 & $<$e-16 \\
ABBmin1	& 899 & 0.0256824 & $<$e-16 & 5 & 0.0007340 & $<$e-16 \\
MomME	& 284 & 0.0210263 & $<$e-16 & 1 & 0.0002030 & $<$e-16 \\
\hline
& \multicolumn{3}{c}{ bcsstm21 } & \multicolumn{3}{c}{ bcsstm22 } \\
\hline
ME	& 5 & 0.0625872 & $<$e-16 & 563 & 0.0050422 & $<$e-16 \\
RelaxME	& 6 & 0.0540492 & $<$e-16 & 72 & 0.0006352 & $<$e-16 \\
CG	& 3 & 0.0179260 & $<$e-16 & 43 & 0.0001700 & $<$e-16 \\
BB1	& 16 & 0.1357834 & $<$e-16 & 122 & 0.0007351 & $<$e-16 \\
ABBmin1	& 12 & 0.0643630 & $<$e-16 & 121 & 0.0004375 & $<$e-16 \\
MomME	& 2 & 0.0180344 & $<$e-16 & 35 & 0.0752466 & $<$e-16 \\
\hline
& \multicolumn{3}{c}{ bibd\_81\_2 } & \multicolumn{3}{c}{ fv1 } \\
\hline
ME	& 1 & 0.0074366 & $<$e-16 & 26 & 2.1586330 & $<$e-16 \\
RelaxME	& 1 & 0.0037551 & $<$e-16 & 20 & 1.0995248 & $<$e-16 \\
CG	& 1 & 0.0073586 & $<$e-16 & 29 & 0.8554080 & $<$e-16 \\
BB1	& 2 & 0.0142210 & $<$e-16 & 35 & 2.0301761 & $<$e-16 \\
ABBmin1	& 2 & 0.0073830 & $<$e-16 & 33 & 0.9284978 & $<$e-16 \\
MomME	& 1 & 0.0038024 & $<$e-16 & 14 & 0.8209626 & $<$e-16 \\
\hline
& \multicolumn{3}{c}{ LF10 } & \multicolumn{3}{c}{ lund\_b } \\
\hline
ME	& 500000 & 0.2910678 & 1.0700e-1 & 52184 & 0.5145678 & $<$e-16 \\
RelaxME	& 11119 & 0.0139615 & $<$e-16 & 1354 & 0.0141224 & $<$e-16 \\
CG	& 47 & 2.2800e-5 & $<$e-16 & 414 & 0.0048636 & $<$e-16 \\
BB1	& 56469 & 0.0345160 & $<$e-16 & 3399 & 0.0237002 & $<$e-16 \\
ABBmin1	& 2911 & 0.0015184 & $<$e-16 & 1730 & 0.0068376 & $<$e-16 \\
MomME	& 576 & 0.0011270 & $<$e-16 & 455 & 0.0049819 & $<$e-16 \\
\hline
\end{tabular}
\end{table}

\begin{table}
\centering
\caption{Numerical results corresponding to Experiment 3.}
\label{tab:32}
\begin{tabular}{c c c c c c c}
\hline		
Method & Iter & Time & NrmG & Iter & Time & NrmG  \\
\hline
& \multicolumn{3}{c}{ mesh1e1 } & \multicolumn{3}{c}{ mesh1em1 } \\
\hline
ME	& 15 & 3.6400e-5 & $<$e-16 & 45 & 8.1800e-5 & $<$e-16 \\
RelaxME	& 15 & 4.3500e-5 & $<$e-16 & 29 & 6.1800e-5 & $<$e-16 \\
CG	& 20 & 2.0200e-5 & $<$e-16 & 32 & 3.1800e-5 & $<$e-16 \\
BB1	& 25 & 3.3000e-5 & $<$e-16 & 51 & 6.3200e-5 & $<$e-16 \\
ABBmin1	& 24 & 2.0600e-5 & $<$e-16 & 52 & 4.0200e-5 & $<$e-16 \\
MomME	& 10 & 3.5600e-5 & $<$e-16 & 20 & 5.0100e-5 & $<$e-16 \\
\hline
& \multicolumn{3}{c}{ mesh1em6 } & \multicolumn{3}{c}{ mesh3em5 } \\
\hline
ME	& 17 & 3.4600e-5 & $<$e-16 & 14 & 0.0005317 & $<$e-16 \\
RelaxME	& 16 & 4.2200e-5 & $<$e-16 & 14 & 0.0007271 & $<$e-16 \\
CG	& 20 & 1.6600e-5 & $<$e-16 & 18 & 0.0004160 & $<$e-16 \\
BB1	& 27 & 2.9900e-5 & $<$e-16 & 24 & 0.0007407 & $<$e-16 \\
ABBmin1	& 26 & 1.9400e-5 & $<$e-16 & 23 & 0.0004774 & $<$e-16 \\
MomME	& 11 & 2.4400e-5 & $<$e-16 & 10 & 0.0005913 & $<$e-16 \\
\hline
& \multicolumn{3}{c}{ mhd3200b } & \multicolumn{3}{c}{ mhd4800b } \\
\hline
ME	& 500000 & 5545.1650430 & $<$e-16 & 500000 & 11668.4923408 & $<$e-16 \\
RelaxME	& 2764 & 20.4574440 & $<$e-16 & 971 & 15.9464792 & $<$e-16 \\
CG	& 5198 & 20.1367672 & $<$e-16 & 6743 & 55.1318183 & $<$e-16 \\
BB1	& 13327 & 92.9306453 & $<$e-16 & 13745 & 212.1669788 & $<$e-16 \\
ABBmin1	& 9094 & 34.0988301 & $<$e-16 & 9816 & 78.6274421 & $<$e-16 \\
MomME	& 17874 & 134.4756059 & $<$e-16 & 22479 & 357.5722142 & $<$e-16 \\
\hline
& \multicolumn{3}{c}{ nos4 } & \multicolumn{3}{c}{ plat362 } \\
\hline
ME	& 1887 & 0.0106610 & $<$e-16 & 500000 & 33.1807635 & $<$e-16 \\
RelaxME	& 154 & 0.0014792 & $<$e-16 & 1178 & 0.0740119 & $<$e-16 \\
CG	& 79 & 0.0002374 & $<$e-16 & 612 & 0.0171572 & $<$e-16 \\
BB1	& 527 & 0.0018996 & $<$e-16 & 7680 & 0.3378808 & $<$e-16 \\
ABBmin1	& 161 & 0.0003997 & $<$e-16 & 4545 & 0.1108690 & $<$e-16 \\
MomME	& 64 & 0.0004834 & $<$e-16 & 1848 & 0.1181252 & $<$e-16 \\
\hline
\end{tabular}
\end{table}

\begin{table}
\centering
\caption{Numerical results corresponding to Experiment 3.}
\label{tab:33}
\begin{tabular}{c c c c c c c}
\hline		
Method & Iter & Time & NrmG & Iter & Time & NrmG  \\
\hline
& \multicolumn{3}{c}{ t2dal\_e } & \multicolumn{3}{c}{ Trefethen\_20 } \\
\hline
ME	& 284 & 9.8335001 & $<$e-16 & 136 & 0.0001072 & $<$e-16 \\
RelaxME	& 33 & 0.7411674 & $<$e-16 & 48 & 7.3500e-5 & $<$e-16 \\
CG	& 93 & 1.0684064 & $<$e-16 & 20 & 1.3800e-5 & $<$e-16 \\
BB1	& 84 & 1.8225858 & $<$e-16 & 88 & 5.3200e-5 & $<$e-16 \\
ABBmin1	& 72 & 0.8267886 & $<$e-16 & 75 & 2.9800e-5 & $<$e-16 \\
MomME	& 35 & 0.8094578 & $<$e-16 & 25 & 5.0900e-5 & $<$e-16 \\
\hline
& \multicolumn{3}{c}{ Trefethen\_20b } & \multicolumn{3}{c}{ Trefethen\_150 } \\
\hline
ME	& 71 & 7.1400e-5 & $<$e-16 & 1600 & 0.0173730 & $<$e-16 \\
RelaxME	& 32 & 6.2400e-5 & $<$e-16 & 227 & 0.0024992 & $<$e-16 \\
CG	& 19 & 9.9000e-6 & $<$e-16 & 114 & 0.0005268 & $<$e-16 \\
BB1	& 67 & 4.7200e-5 & $<$e-16 & 403 & 0.0032595 & $<$e-16 \\
ABBmin1	& 60 & 2.7600e-5 & $<$e-16 & 294 & 0.0012258 & $<$e-16 \\
MomME	& 20 & 3.2600e-5 & $<$e-16 & 82 & 0.0009541 & $<$e-16 \\
\hline
& \multicolumn{3}{c}{ Trefethen\_300 } & \multicolumn{3}{c}{ Trefethen\_500 } \\
\hline
ME	& 3674 & 0.1375534 & $<$e-16 & 6598 & 0.7816970 & $<$e-16 \\
RelaxME	& 386 & 0.0159382 & $<$e-16 & 421 & 0.0402854 & $<$e-16 \\
CG	& 178 & 0.0029016 & $<$e-16 & 242 & 0.0116076 & $<$e-16 \\
BB1	& 450 & 0.0128855 & $<$e-16 & 765 & 0.0508177 & $<$e-16 \\
ABBmin1	& 411 & 0.0060859 & $<$e-16 & 498 & 0.0183909 & $<$e-16 \\
MomME	& 119 & 0.0049866 & $<$e-16 & 155 & 0.0138524 & $<$e-16 \\
\hline
& \multicolumn{3}{c}{ Trefethen\_2000 } & \multicolumn{3}{c}{ fv2 } \\
\hline
ME	& 32102 & 164.8055842 & $<$e-16 & 24 & 2.1007780 & $<$e-16 \\
RelaxME	& 1175 & 4.0813674 & $<$e-16 & 20 & 1.1607601 & $<$e-16 \\
CG	& 545 & 0.9439436 & $<$e-16 & 29 & 0.8901184 & $<$e-16 \\
BB1	& 1598 & 4.3126404 & $<$e-16 & 35 & 2.0813600 & $<$e-16 \\
ABBmin1	& 1165 & 2.0048808 & $<$e-16 & 33 & 0.9790256 & $<$e-16 \\
MomME	& 323 & 1.1291722 & $<$e-16 & 14 & 0.9435173 & $<$e-16 \\
\hline
& \multicolumn{3}{c}{ Kuu } & \multicolumn{3}{c}{ Muu } \\
\hline
ME	& 29513 & 1368.3078077 & $<$e-16 & 41 & 1.9187869 & $<$e-16 \\
RelaxME	& 1144 & 35.2583332 & $<$e-16 & 21 & 0.6481658 & $<$e-16 \\
CG	& 712 & 11.1337830 & $<$e-16 & 32 & 0.5324427 & $<$e-16 \\
BB1	& 1376 & 42.0703310 & $<$e-16 & 39 & 1.2222348 & $<$e-16 \\
ABBmin1	& 1122 & 17.2900812 & $<$e-16 & 37 & 0.6568082 & $<$e-16 \\
MomME	& 413 & 12.7089510 & $<$e-16 & 16 & 0.5087381 & $<$e-16 \\
\hline
\end{tabular}
\end{table}

\subsection{Experiment 4: grayscale images smoothing}
Image smoothing is an important step in image processing, acting as a digital filter that cleans up images, see \cite{Tomasi1998}. A simple optimization model for smoothing a grayscale image $Y\in\mathbb{R}^{n\times m}$ is given by
\begin{equation}
\min_{X = [X_{i,j}] \in \mathbb{R}^{n \times m}} \frac{1}{2} \|X-Y\|_F^2 + \frac{\lambda}{2}\sum_{i=2}^{n} \sum_{j=1}^{m} (X_{i,j} - X_{i-1,j})^2 + \frac{\lambda}{2} \sum_{i=1}^{n} \sum_{j=2}^{m} (X_{i,j} - X_{i,j-1})^2, \label{img_model}
\end{equation}
where $\lambda>0$ is the smoothing parameter, which regularizes the objective function by penalizing large intensity variations between adjacent pixels both vertically and horizontally.  Thus, adjusting $\lambda$, we directly control the intensity of the smoothing effect. Notice that the objective function of problem \eqref{img_model}
is quadratic and strictly convex.

For this experiment, we consider the images \emph{boat.tiff} ($512\times 512$) 
and \emph{male.tiff} ($1024\times 1024$) 
which are available in the USC--SIPI Image Database, Miscellaneous volume\footnote{The images can be downloaded from \url{https://sipi.usc.edu/database/}}. In addition,  we solve the optimization problem \eqref{img_model} with $\lambda = 1$ and $\lambda=100$, considering the following iterative methods: BB1, ABBmin1, RelaxME, ME, CG, and MomME. The initial point for all the methods was taken as $X_0 = Y$, that is, the initial point corresponds to the image that we want to smooth. The tolerance used to stop the methods was $1e$-$7$. Table \ref{tab:5} summarizes the numerical results obtained by the methods for each image. Figure \ref{imageboats} shows the images (global minimizers) obtained by each method for {\em{boat.tiff}}, $\lambda=1$
and for {\em{male.tiff}}, $\lambda=100$.
For {\em{boat.tiff}} and both
$\lambda=1$ and $\lambda=100$, CG is the quickest method. For
image {\em{male.tiff}} of larger
size $1024 \times 1024$, our method
MomME is the quickest, by far when
$\lambda=100$. Therefore on this
experiment too, MomME is competitive
with the other optimizers for large scale problems.
Of course, the larger $\lambda$
the more the smoothed image is 
different from the original image.

\begin{table}
\centering
\caption{Numerical results corresponding to Experiment 4.}
\label{tab:5}
\begin{tabular}{c c c c c c c }
\hline		
Method & Iter & Time & NrmG & Iter & Time & NrmG \\
\hline
& \multicolumn{3}{c}{ boat ($512\times 512$), $\lambda=1$ } & \multicolumn{3}{c}{ boat ($512\times 512$), $\lambda=100$ } \\
\hline		
RelaxME	    &	23	&	2.74	&	$<10^{-6}$	& 217 & 48.83 & $<10^{-6}$\\
MomME	    &	14	&	2.63	&	$<10^{-6}$	& 155 &40.38  &$<10^{-6}$\\
ME	    &	23	&	2.37	&	$<10^{-6}$	& 1615 & 357.69  &$<10^{-6}$\\
CG	    &	29	&	1.25	&	$<10^{-6}$	& 310 & 19.62  &$<10^{-6}$\\
BB1	    &	34	&	4.51	&	$<10^{-6}$	& 404  &34.31  &$<10^{-6}$\\
ABBmin1	&	34	&	1.70	&	$<10^{-6}$	& 399 & 32.98  &$<10^{-6}$\\
\hline
& \multicolumn{3}{c}{ male ($1024\times 1024$), $\lambda=1$ }  
& \multicolumn{3}{c}{ male ($1024\times 1024$), $\lambda=100$ }
\\
\hline
RelaxME	    &	23	&	11.65	&	$<10^{-6}$ & 232	&232.00  &  $<10^{-6}$\\
MomME	    &	14	&	7.75	&	$<10^{-6}$	& 160 &140.62  &$<10^{-6}$\\
ME	    &	24	&	17.88	&	$<10^{-6}$	& 1719 &1711.10  &$<10^{-6}$\\
CG	    &	30	&	16.24	&	$<10^{-6}$	& 322 & 172.61 &$<10^{-6}$\\
BB1	    &	38	&	8.92	&	$<10^{-6}$	& 418 &199.59  &$<10^{-6}$\\
ABBmin1	&	34	&	7.82	&	$<10^{-6}$ &	386 & 177.88 &  $<10^{-6}$\\
\hline
\end{tabular}
\end{table}

\begin{figure}
\centering
\centering
      \begin{tabular}{c} \includegraphics[height=0.22\textheight,width=0.5\textwidth]{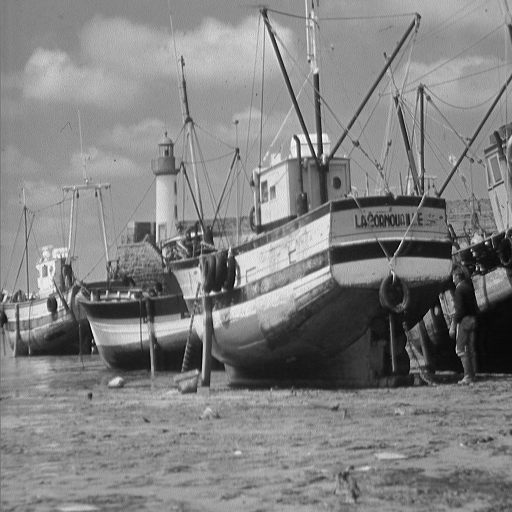}
\end{tabular}
      \begin{tabular}{cc} \includegraphics[height=0.22\textheight,width=0.5\textwidth]{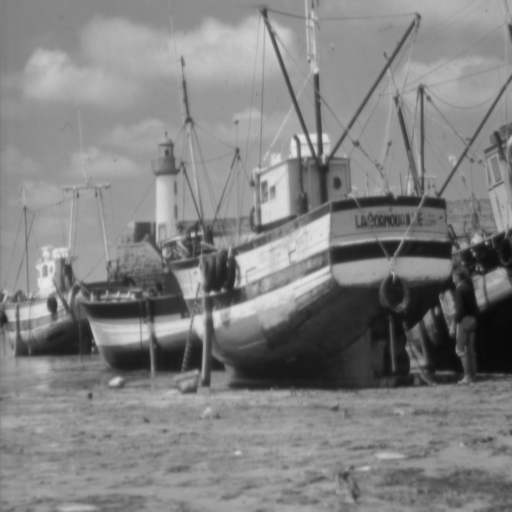}
        &
\includegraphics[height=0.22\textheight,width=0.5\textwidth]{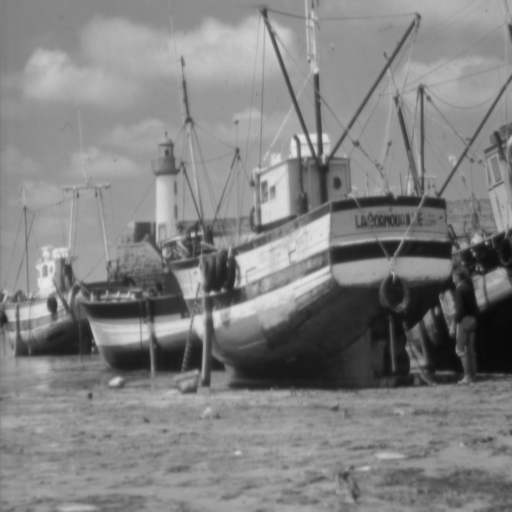}\\
ABBmin1 & BB1 \\
\includegraphics[height=0.22\textheight,width=0.5\textwidth]{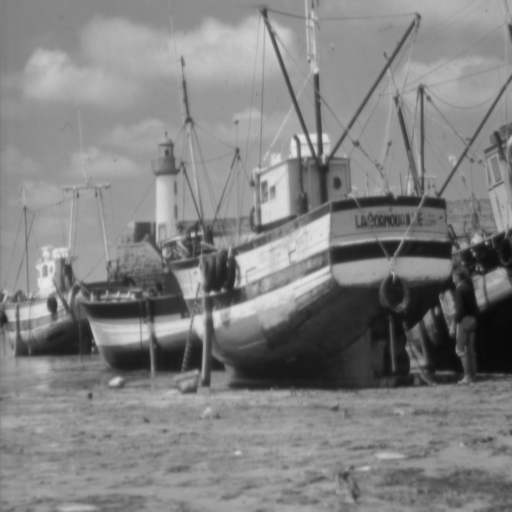}
        &
\includegraphics[height=0.22\textheight,width=0.5\textwidth]{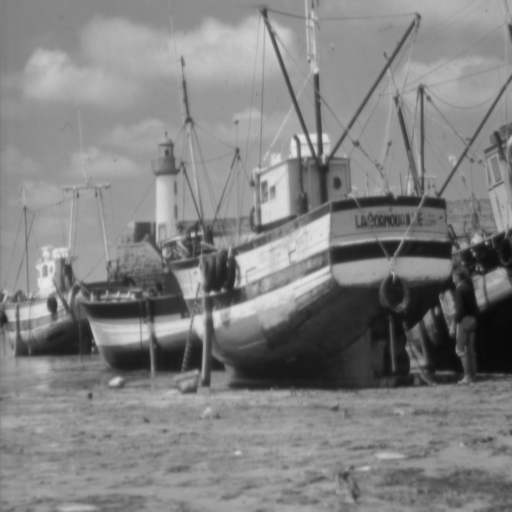}\\
ME & RelaxME\\
\includegraphics[height=0.22\textheight,width=0.5\textwidth]{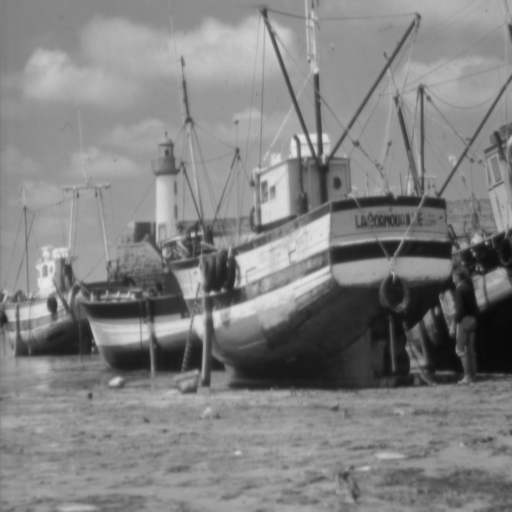}
        &
\includegraphics[height=0.22\textheight,width=0.5\textwidth]{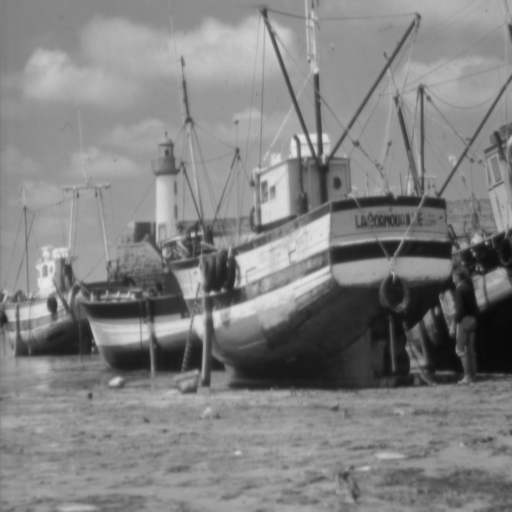}\\
MomME & CG
\end{tabular}
\caption{Original image {\em{boat.tiff}} (upper plot) and smoothed images
with ABBmin1, BB1, ME, RelaxME, MomME, and CG for $\lambda=1$.}\label{imageboats}
\end{figure}

\begin{figure}
\centering
\centering
      \begin{tabular}{c} \includegraphics[height=0.22\textheight,width=0.5\textwidth]{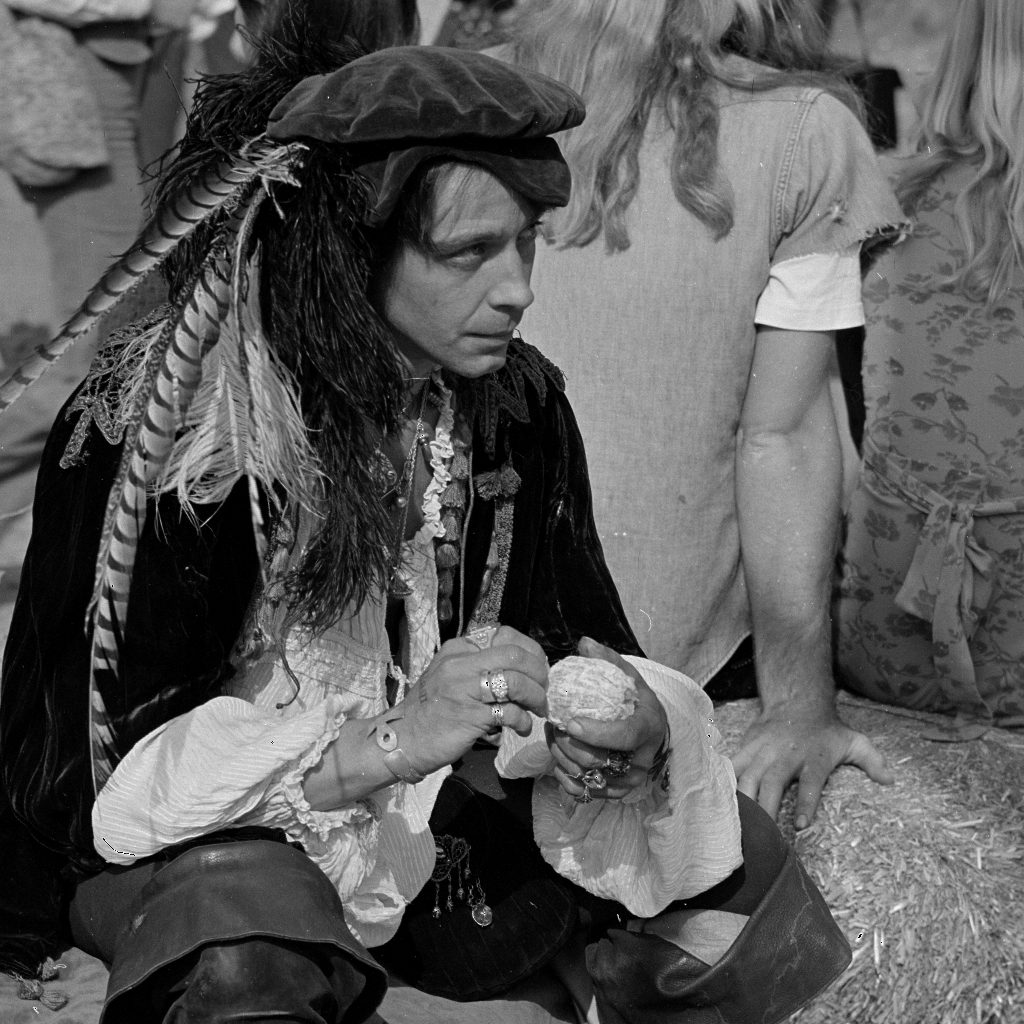}
\end{tabular}
      \begin{tabular}{cc} \includegraphics[height=0.22\textheight,width=0.5\textwidth]{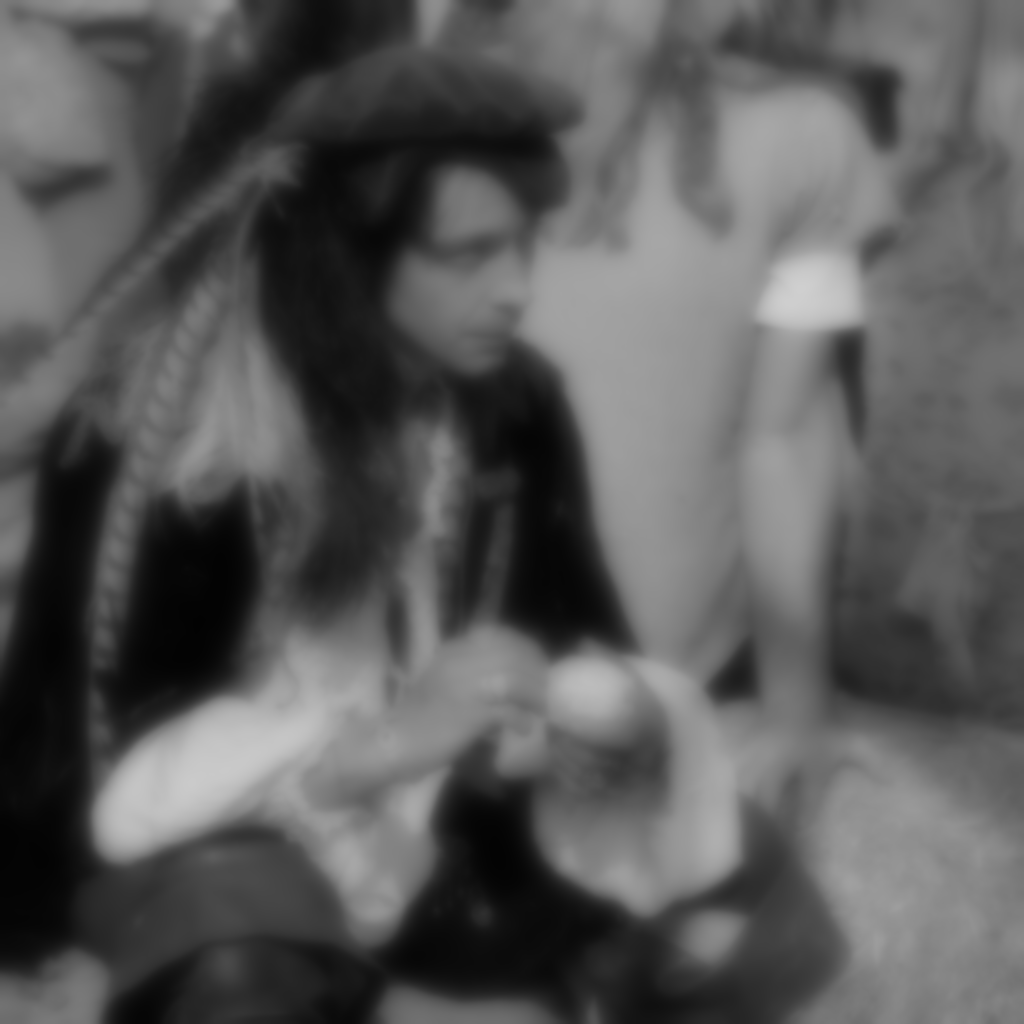}
        &
\includegraphics[height=0.22\textheight,width=0.5\textwidth]{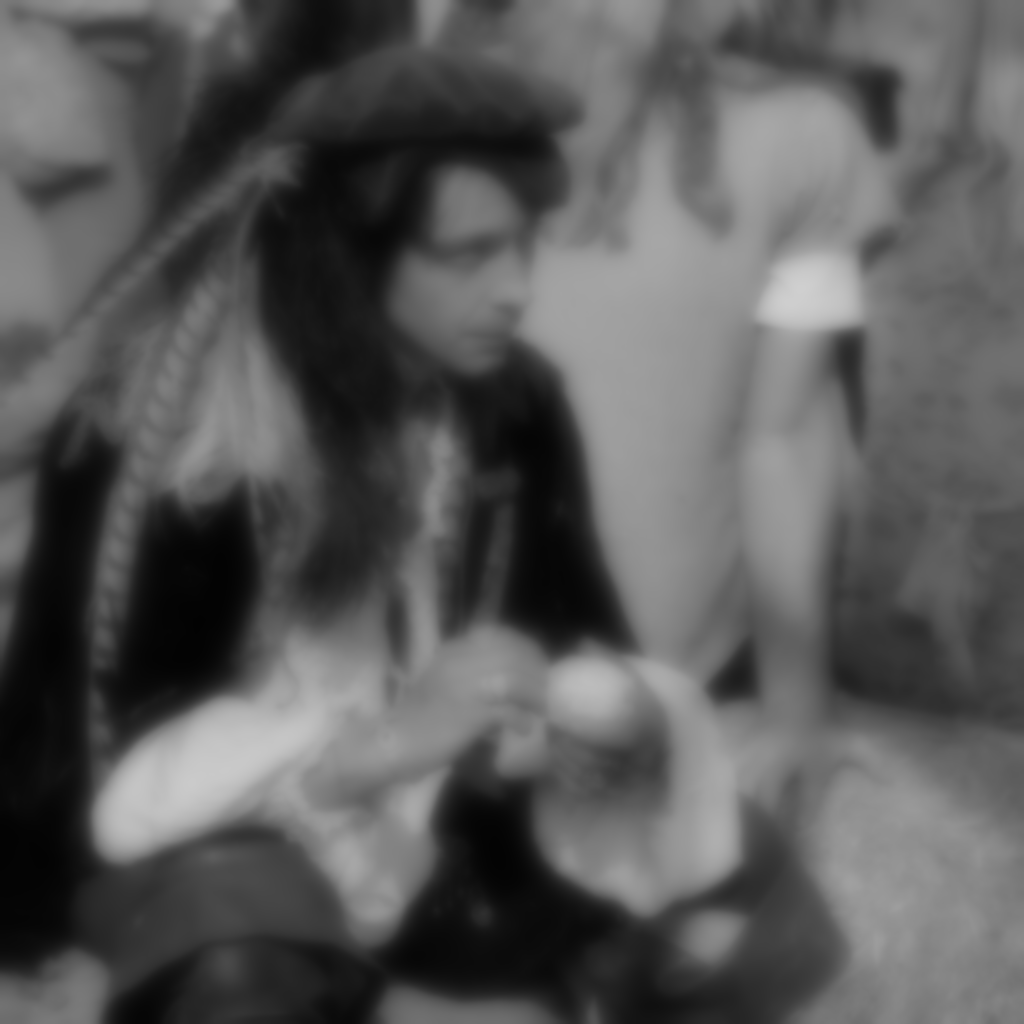}\\
ABBmin1 & BB1 \\
\includegraphics[height=0.22\textheight,width=0.5\textwidth]{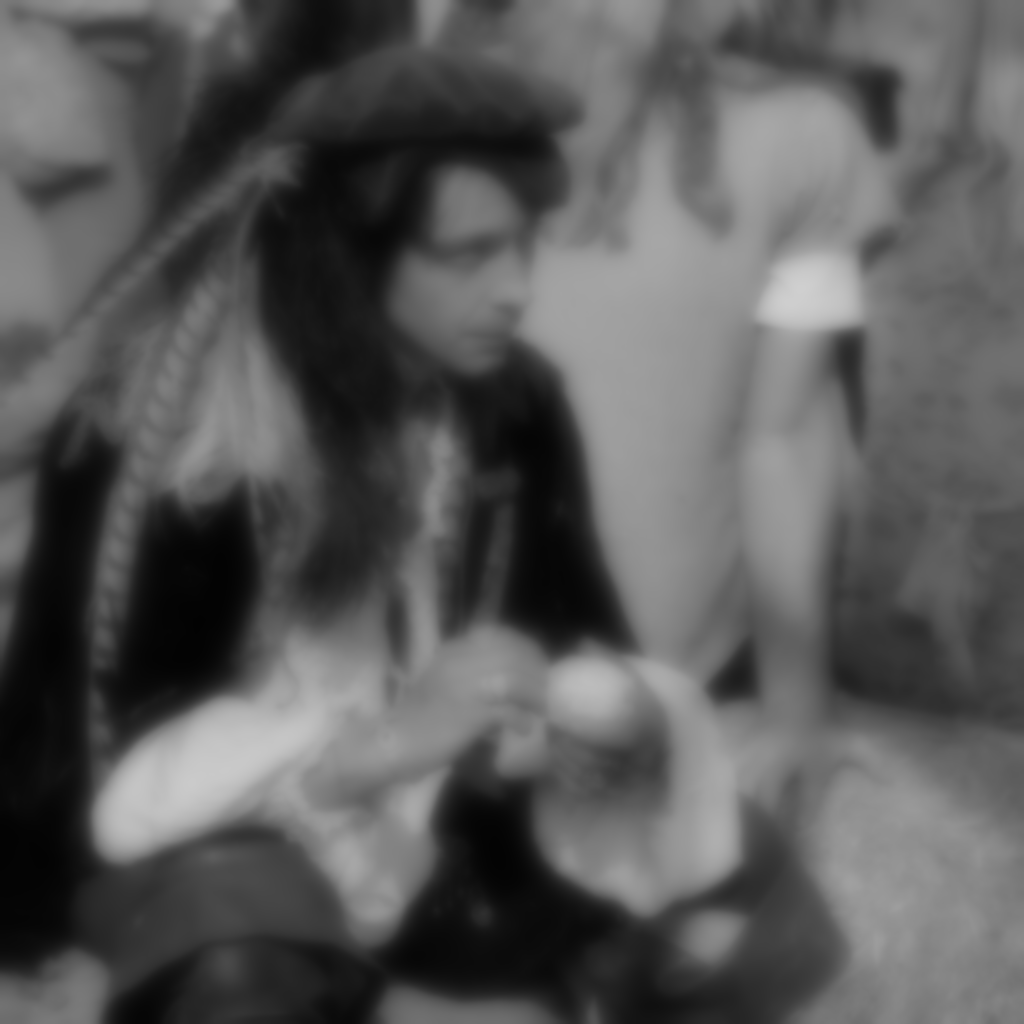}
        &
\includegraphics[height=0.22\textheight,width=0.5\textwidth]{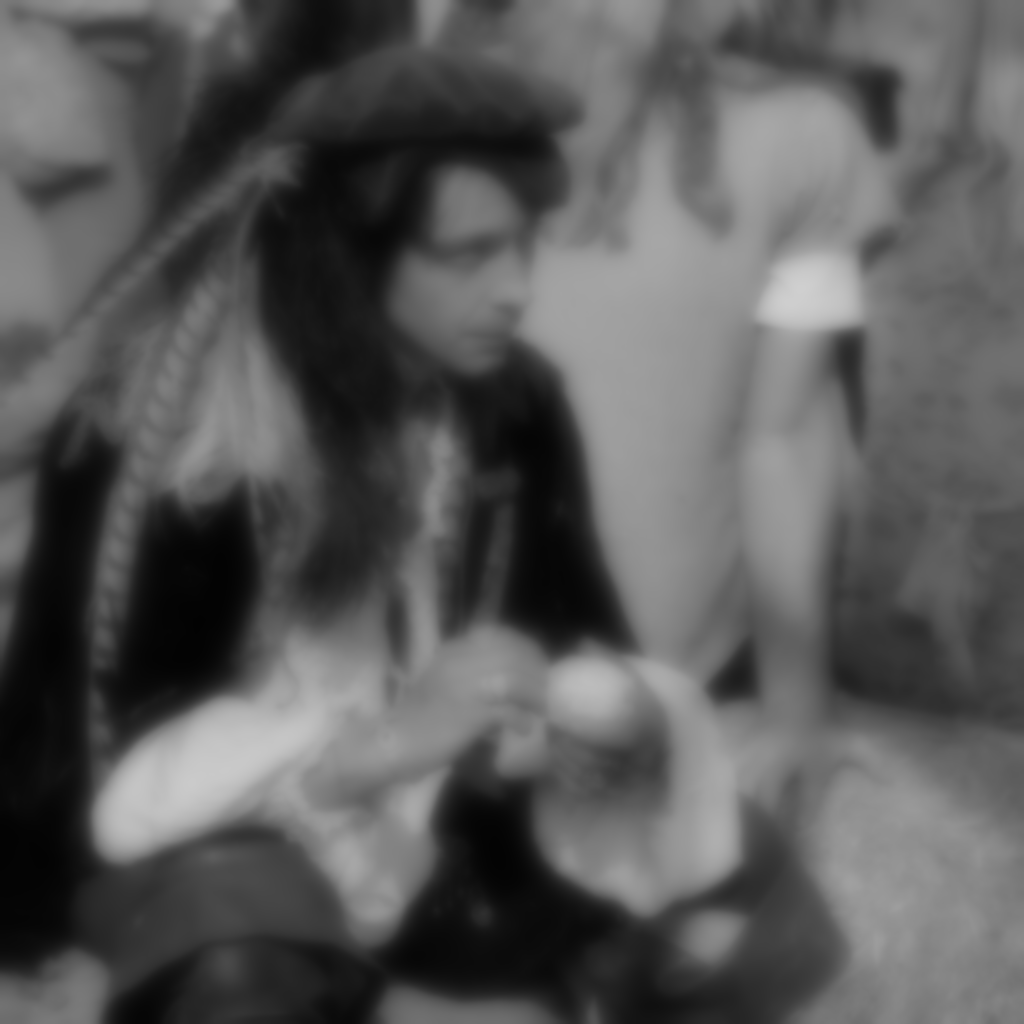}\\
ME & RelaxME\\
\includegraphics[height=0.22\textheight,width=0.5\textwidth]{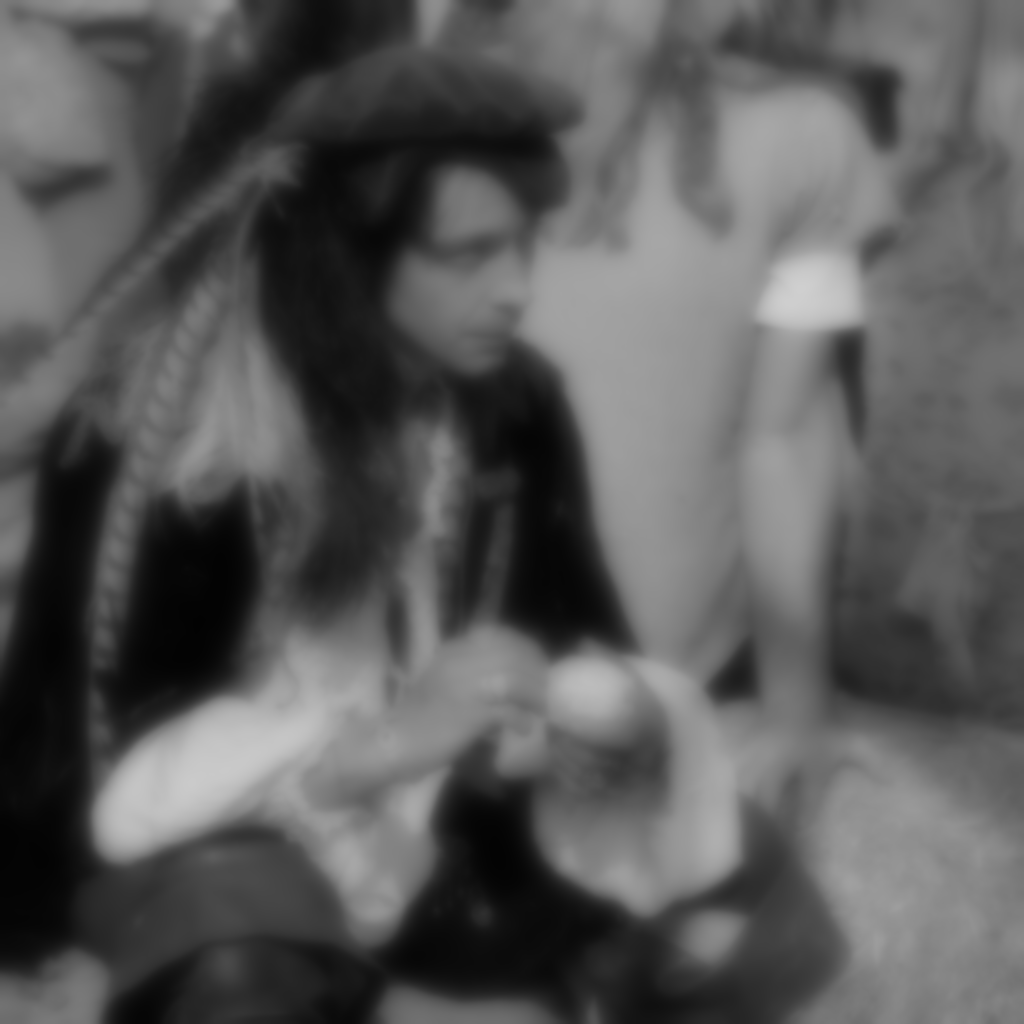}
        &
\includegraphics[height=0.22\textheight,width=0.5\textwidth]{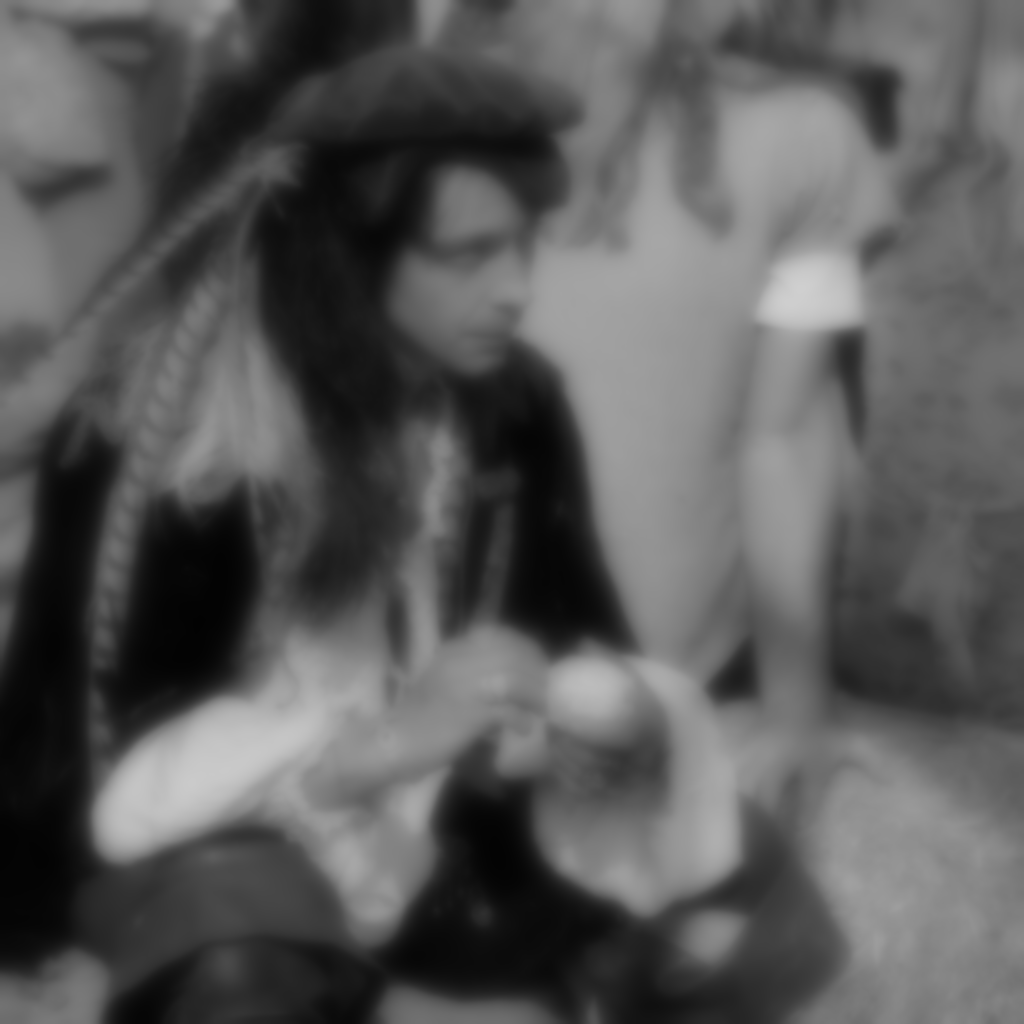}\\
MomME & CG
\end{tabular}
\caption{Original image {\em{male.tiff}} (upper plot) and smoothed images
with ABBmin1, BB1, ME, RelaxME, MomME, and CG for $\lambda=100$.}\label{imagesmale}
\end{figure}

\subsection{Experiment 5: large scale matrices with two different eigenvalues}

To check fininte convergence in one iteration for
ME, RelaxME, and MomME, we build an
$n \times n$
definite positive matrix
$A$ with two different eigenvalues as follows
and take $n \in \{5000, 10000, 20000\}$.
We take $D$ diagonal with 
$D=\mbox{diag}(d_1,d_2,\ldots,d_n)$ where
$d_i=1$ for $i=1,\ldots,n/2$
and $d_i=2$ for $i=1+n/2,\ldots,n$.
We then build a random orthogonal matrix
$P$ and matrix A as 
$A=PDP^{T}$.
To build $P$, we take
$B$ an $n \times n$ matrix whose entries
are realizations of independent
random variables with uniform distribution
on $[0,1]$. Then we compute
the QR decomposition $B=QR$ of $B$
and take $P=Q$.
For these three matrices $A$, generated
as we have just explained and taking
the three values $5000$, $10000$, $20000$ 
of the matrix size,
we run 10 times BB1, MomME, ME, RelaxME, 
ABBmin1, and CG
and report in Table \ref{finiteconvt}  the mean number of
iterations, the mean CPU time, and the mean
norm of the gradient at termination, where the means
are computed over the 10 runs.

\begin{table}
\centering
\caption{Numerical results corresponding to Experiment 5.}
\label{tab:5}
\begin{tabular}{c c c c }
\hline		
Method & Iter & Time & NrmG \\
\hline
\multicolumn{4}{c}{$n=5000$}  \\
\hline		
BB1	    &	9	&	0.11	&	$<10^{-6}$	\\
MomME	&	1	&0.03		&	$<10^{-6}$	\\
ABBmin1	&	9	&	0.09	&	$<10^{-6}$	\\
RelaxME	&	1	& 0.02		&	$<10^{-6}$	\\
ME	&	1	&	0.02	&	$<10^{-6}$	\\
CG	&2	&	0.03	&	$<10^{-6}$	\\
\hline
\multicolumn{4}{c}{ $n=10000$ }  \\
\hline		
BB1	    &	9	& 0.41		&	$<10^{-6}$	\\
MomME	&	1	&	0.08	&	$<10^{-6}$	\\
ABBmin1	&	9	&	0.37	&	$<10^{-6}$	\\
RelaxME	&	1	&	0.08	&	$<10^{-6}$	\\
ME	&	1	&	0.13	&	$<10^{-6}$	\\
CG	&2	&	0.13	&	$<10^{-6}$	\\
\hline
\multicolumn{4}{c}{ $n=20000$ }  \\
\hline		
BB1	    &	9	&	2.34	&	$<10^{-6}$	\\
MomME	&	1	&	0.50	&	$<10^{-6}$	\\
ABBmin1	&	9	&	1.84	&	$<10^{-6}$	\\
RelaxME	&	1	&	0.46	&	$<10^{-6}$	\\
ME	&	1	&	0.93	&	$<10^{-6}$	\\
CG	&2	&	0.73	&	$<10^{-6}$	\\
\hline
\end{tabular}
\caption{Convergence in one iteration of ME, RelaxME, and
MomME for matrices of size $n$ having 2 distinct eigenvalues.
Mean number of iterations, mean CPU time, and mean norm of the gradient at termination for BB1, MomME, ME, RelaxME, 
ABBmin1, and CG.}\label{finiteconvt}
\end{table}

In Table \ref{finiteconvt}, in accordance with Theorem 
\ref{finiteconv}, we observe the convergence
in one iteration of ME, RelaxME, and MomME
for three instances where matrix $A$
has two distinct eigenvalues.
We also see that the mean CPU time
with at least one of these methods is the smallest
among all optimizers.
In particular, ME, Relax ME, and MomME are competitive
with CG which requires 2 iterations on these instances.

\section*{Appendix}

\subsection*{New proof of \eqref{ineqfriststep} when $g_k$ and $r_k$ are linearly independent}

In this section, we provide a
new proof of \eqref{ineqfriststep}
when $g_k$ and $r_k$ are linearly independent.
This proof uses 
Kantorovich inequality.
Let us define the auxiliary sequence of real numbers $(b_k)$ by $b_k := \frac{\|\nabla f(z_k)\|^2}{\nabla f(z_k)^{T}A\nabla f(z_k)}$, where
$z_k = x_k - \frac{t_k}{2}g_k$, for all $k\geq 0$. Using the Taylor expansion of $f$, we have
\begin{eqnarray}
f(z_k - b_k\nabla f(z_k)) & = & 
f(z_k)-b_k \|\nabla f(z_k)\|^2 + \frac{b_k^2}{2}\nabla f(z_k)^T A \nabla f(z_k)^T \nonumber\\ 
& = & 
f(z_k)-\frac{b_k}{2} \|\nabla f(z_k)\|^2. \label{firstrelapp}
\end{eqnarray}

Using this relation and the optimality of the step--sizes $(\alpha_k,\beta_k)$ we have that
\begin{eqnarray}
f(\tilde x_{k+1}) & \leq &  f(x_k + \alpha_kg_k + \beta_k r_k) \nonumber \\
& \leq &  f\left(x_k - \frac{1}{2}(t_k + b_k)g_k - \frac{1}{2}b_k r_k\right) \nonumber \\
& = &  f\left(z_k - \frac{b_k}{2}(g_k + r_k)\right) \nonumber \\
& = &  f\left(z_k - \frac{b_k}{2}(\nabla f(x_k) + \nabla f(y_k))\right) \nonumber \\
& = &  f\left(z_k - \frac{b_k}{2}(A x_k + A(x_k-t_k g_k)-2b)\right) \nonumber \\
& = &  f(z_k - b_k\nabla f(z_k))  \nonumber \\
& \stackrel{\eqref{firstrelapp}}{=} &  f(z_k) - \frac{b_k}{2}\|\nabla f(z_k)\|^2. \label{eq1_theorem1}
\end{eqnarray}
Subtracting $f(x_{*})$ from both sides of \eqref{eq1_theorem1}, we get
\begin{equation}
f(\tilde x_{k+1}) - f(x_{*}) \leq  f(z_k)-f(x_{*}) - \frac{b_k}{2}\|\nabla f(z_k)\|^2. \label{eq2_theorem1}
\end{equation}

Now, using Kantorovich inequality with $v = \nabla f(z_k) \neq 0$ and noting that $\nabla f(z_k) = Az_k - b = A(z_k - x_{*})$, we obtain
\begin{equation}
\frac{1}{2}b_k \|\nabla f(z_k)\|^2 \geq {\eta}(f(z_k) - f(x_{*})), \label{eq3_theorem1}
\end{equation}
where ${\eta} := \frac{4\lambda_{1}\lambda_{n}}{(\lambda_{1} + \lambda_{n})^2}>0$.
Thus, in view of \eqref{eq2_theorem1} and \eqref{eq3_theorem1}, we arrive at
\begin{eqnarray}
&&f(\tilde x_{k+1}) - f(x_{*})  \leq  \left( 1 - {\eta} \right)(f(z_k) - f(x_{*})) \nonumber \\
&  &\leq\left( \frac{\lambda_{n} - \lambda_{1}}{\lambda_{n} + \lambda_{1}} \right)^2(f(z_k) - f(x_{*})). \label{eq4_theorem1}
\end{eqnarray}
On the other hand, from the definition of $z_k$, we have
\begin{equation}
f(z_k) - f(x_{*}) = (f(x_k) - f(x_{*})) - \frac{t_k}{4}\|g_k\|^2. \label{eq5_theorem1}
\end{equation}
Again, using Kantorovich's inequality in the last term of \eqref{eq5_theorem1} we have
$$ \frac{t_k}{4}\|g_k\|^2 \geq \bar{\eta}(f(x_k) - f(x_{*})). $$
Substituting this last result into \eqref{eq5_theorem1}, we obtain
\begin{equation}
f(z_k) - f(x_{*}) \leq (1-\bar{\eta}) (f(x_k) - f(x_{*})) \leq \left( \frac{\lambda_{n} - \lambda_{1}}{\lambda_{n} + \lambda_{1}} \right)^2 (f(x_k) - f(x_{*})). \label{eq6_theorem1}
\end{equation}
Finally, combining \eqref{eq4_theorem1} with \eqref{eq6_theorem1}, we arrive at
\begin{eqnarray}
f(\tilde x_{k+1}) - f(x_{*}) &\leq& \left( \frac{\lambda_{n} - \lambda_{1}}{\lambda_{n} + \lambda_{1}} \right)^4 (f(x_k) - f(x_{*})), 
\end{eqnarray}
which achieves the proof of \eqref{ineqfriststep}.
$\hfill \square$

\end{document}